\documentclass[12pt]{amsart}

\usepackage{amsxtra,amssymb,amsthm,amsmath,amscd,url,listings}

\usepackage[utf8]{inputenc}
\usepackage{eucal}
\usepackage{fullpage}
\usepackage{scrtime}
\usepackage{csquotes}
\usepackage[colorlinks]{hyperref}
\usepackage[]{graphicx}
\usepackage{subcaption}
\usepackage{float}
\def\setminus{\mathchoice
    {\mathbin{\vrule height .62ex width 1.61ex depth -.38ex}}
    {\mathbin{\vrule height .62ex width 1.61ex depth -.38ex}}
    {\mathbin{\vrule height .50ex width 0.85ex depth -.28ex}}
    {\mathbin{\vrule height .20ex width 0.570ex depth -.24ex}}
}

\DeclareMathSymbol{A}{\mathalpha}{operators}{`A}%
\DeclareMathSymbol{B}{\mathalpha}{operators}{`B}%
\DeclareMathSymbol{C}{\mathalpha}{operators}{`C}%
\DeclareMathSymbol{D}{\mathalpha}{operators}{`D}%
\DeclareMathSymbol{E}{\mathalpha}{operators}{`E}%
\DeclareMathSymbol{F}{\mathalpha}{operators}{`F}%
\DeclareMathSymbol{G}{\mathalpha}{operators}{`G}%
\DeclareMathSymbol{H}{\mathalpha}{operators}{`H}%
\DeclareMathSymbol{I}{\mathalpha}{operators}{`I}%
\DeclareMathSymbol{J}{\mathalpha}{operators}{`J}%
\DeclareMathSymbol{K}{\mathalpha}{operators}{`K}%
\DeclareMathSymbol{L}{\mathalpha}{operators}{`L}%
\DeclareMathSymbol{M}{\mathalpha}{operators}{`M}%
\DeclareMathSymbol{N}{\mathalpha}{operators}{`N}%
\DeclareMathSymbol{O}{\mathalpha}{operators}{`O}%
\DeclareMathSymbol{P}{\mathalpha}{operators}{`P}%
\DeclareMathSymbol{Q}{\mathalpha}{operators}{`Q}%
\DeclareMathSymbol{R}{\mathalpha}{operators}{`R}%
\DeclareMathSymbol{S}{\mathalpha}{operators}{`S}%
\DeclareMathSymbol{T}{\mathalpha}{operators}{`T}%
\DeclareMathSymbol{U}{\mathalpha}{operators}{`U}%
\DeclareMathSymbol{V}{\mathalpha}{operators}{`V}%
\DeclareMathSymbol{W}{\mathalpha}{operators}{`W}%
\DeclareMathSymbol{X}{\mathalpha}{operators}{`X}%
\DeclareMathSymbol{Y}{\mathalpha}{operators}{`Y}%
\DeclareMathSymbol{Z}{\mathalpha}{operators}{`Z}%

\renewcommand{\leq}{\leqslant}
\renewcommand{\geq}{\geqslant}

\newcommand{\Cc}{\mathbf{C}}
\newcommand{\Oo}{\mathbf{O}}
\newcommand{\Oc}{\mathcal{O}}

\newcommand{\Zz}{\mathbf{Z}}

\newcommand{\Rr}{\mathbf{R}}

\newcommand{\Ss}{\mathbf{S}}

\newcommand{\Hh}{\mathbf{H}}
\newcommand{\Qq}{\mathbf{Q}}

\newcommand{\Ff}{\mathbf{F}}

\newcommand{\mmu}{\boldsymbol{\mu}}
\newcommand{\expect}{\mathbf{E}}

\newcommand{\mods}[1]{\,(\mathrm{mod}\,{#1})}

\DeclareMathOperator{\Kl}{Kl}

\DeclareMathOperator{\Imag}{Im}

\DeclareMathOperator{\ind}{\kappa}

\DeclareMathOperator{\Tr}{tr}

\DeclareMathOperator{\disc}{disc}

\renewcommand{\rho}{\varrho}

\DeclareMathOperator{\SL}{SL}

\DeclareMathOperator{\Sp}{Sp}

\DeclareMathOperator{\SU}{SU}
\DeclareMathOperator{\Un}{U}
\DeclareMathOperator{\USp}{USp}

\DeclareMathSymbol{\gena}{\mathord}{letters}{"3C}
\DeclareMathSymbol{\genb}{\mathord}{letters}{"3E}

\theoremstyle{plain}
\newtheorem{theorem}{Theorem}[section]
\newtheorem{lemma}[theorem]{Lemma}
\newtheorem{corollary}[theorem]{Corollary}

\newtheorem{proposition}[theorem]{Proposition}

\theoremstyle{remark}

\theoremstyle{definition}

\newtheorem*{question}{Question}
\newtheorem{example}[theorem]{Example}
\newtheorem{remark}[theorem]{Remark}

\newcommand{\mcF}{\mathcal{F}}
\newcommand{\mcG}{\mathcal{G}}

\renewcommand{\geq}{\geqslant}
\renewcommand{\leq}{\leqslant}

\begin{document}

\title{Ultra-short sums of trace functions}

\author{E. Kowalski}
\author{T. Untrau}
\address{ETH Z\"urich -- D-MATH\\
  R\"amistrasse 101\\
  8092 Z\"urich\\
  Switzerland} 
\email{kowalski@math.ethz.ch}
\address{Université de Bordeaux, CNRS, Bordeaux INP, IMB, UMR 5251, F-33400 \\ Talence, France} 
\email{theo.untrau@math.u-bordeaux.fr}

\date{\today,\ \thistime} 



\begin{abstract}
  We generalize results of Duke, Garcia, Hyde, Lutz and others on the
  distribution of sums of roots of unity related to Gaussian periods
  to obtain equidistribution of similar sums over zeros of arbitrary
  integral polynomials. We also interpret these results in terms of
  trace functions, and generalize them to higher rank trace functions.
\end{abstract}

\maketitle


\section{Introduction}

The motivation for this work lies in papers of Garcia, Hyde and
Lutz~\cite{ghl} and Duke, Garcia and Lutz~\cite{dgl}, recently
generalized by Untrau~\cite{untrau} in a number of ways, which
considered the distribution properties of certain finite sums of roots
of unity which are related to Gaussian periods and to
``supercharacters'' of finite groups.

We interpret these sums as examples of sums of trace functions over
certain \emph{bounded} finite sets.  From this point of view, this
study is a complement to results concerning sums of trace functions
with growing length modulo a prime~$p$ (for instance, the paper of
Perret-Gentil~\cite{p-g} for sums of length roughly up to $\log p$, or
that of Fouvry, Kowalski, Michel, Raju, Rivat and
Soundararajan~\cite{sliding} for sums of length slightly above
$\sqrt{p}$, and that of Kowalski and Sawin~\cite{ks} for sums of
length proportional to~$p$).


The range of summation will be taken to be more general than an
interval, and despite the simplicity of the setting, one obtains some
interesting equidistribution results.

Here is a simple illustration of our statements. More general versions
will be proved in Sections~\ref{sec-proofs1},~\ref{section
  multiplicative} and~\ref{sec-proofs2}. We recall the definition
$$
\Kl_2(a;q)=\frac{1}{\sqrt{q}}\sum_{x\in\Ff_q^{\times}}
e\Bigl(\frac{ax+\bar{x}}{q}\Bigr),\quad\quad
e(z)=e^{2i\pi z},
$$
of the normalized Kloosterman sums modulo a prime number~$q$.

\begin{theorem}[Ultra-short sums of additive characters and
  Kloosterman sums]\label{th-1}
  Let $g\in~\Zz[X]$ be a fixed monic polynomial of degree~$d\geq
  1$. For any field~$K$, denote by $Z_g(K)$ the set of zeros of~$g$
  in~$K$, and put $Z_g=Z_g(\Cc)$. Let $K_g= \Qq(Z_g)$ be the splitting
  field of $g$.
  \par
  \emph{(1)} As $q\to +\infty$ among prime numbers unramified and
  totally split in $K_g$, the sums
  $$
  \sum_{x\in Z_g(\Ff_q)}e\Bigl(\frac{ax}{q}\Bigr)
  $$
  parameterized by $a\in\Ff_q$ become equidistributed in~$\Cc$ with
  respect to some explicit probability measure~$\mu_g$.
  \par
  \emph{(2)} Suppose that $0\notin Z_g$. As $q\to +\infty$ among prime
  numbers unramified and totally split in $K_g$, the sums
  $$
  \sum_{x\in Z_g(\Ff_q)}\Kl_2(ax;q)
  $$
  parameterized by $a\in\Ff_q$ become equidistributed in~$\Cc$ with
  respect to the measure which is the law of the sum of $d$ independent
  Sato--Tate random variables.
\end{theorem}

\begin{example}
  (1) The case considered in the previous papers that we mentioned is
  that of $g=X^d-1$ for some integer $d\geq 1$, in which case $Z_g$ is
  the set of $d$-th roots of unity and the primes involved are the prime
  numbers congruent to~$1$ modulo~$d$. (In fact, these references
  consider more generally the sums above for $q$ a power of an odd prime
  $\equiv 1\mods{d}$, and we will also handle this case.)
  \par
  (2) The measure $\mu_{g}$ can be described relatively explicitly,
  and depends on the additive relations (with integral coefficients)
  satisfied by the zeros of $g$. We will discuss this in more detail
  below, but ``generically'', we will see that $\mu_g$ is just the law
  of the sum $X_1+\cdots+X_d$ of $d$ independent random variables each
  uniformly distributed on the unit circle.  However, more interesting
  measures also arise, for instance for $g=X^{\ell}-1$ where $\ell$ is
  a prime number, the measure $\mu_{X^{\ell}-1}$ is the image by the
  map
  $$
  (z_1,\ldots,z_{\ell-1})\mapsto
  z_1+\cdots+z_{\ell-1}+\frac{1}{z_1\cdots z_{\ell-1}}
  $$
  of the uniform (Haar) probability measure on
  $(\Ss^1)^{\ell-1}$. Figure \ref{relations_or_not} below illustrates
  two examples. In the case of the polynomial $X^3 + 2X^2 + 3$, one
  can show that there are no non-trivial additive relations between
  the zeros of $g$, whereas in the case of the polynomial
  $X^3 + X+ 3$, there is clearly the relation given by the sum of the
  roots which equals zero (because the coefficient of $X^2$ is
  zero). We see that this difference between their module of additive
  relations translates into different limiting measures $\mu_g$ for
  the associated sums of additive characters. Since these two
  polynomials have Galois group $\mathfrak{S}_3$ over $\Qq$, these
  pictures will be fully explained in Section \ref{sec-examples},
  Example 2.
  
  \begin{figure}
    \centering
    \begin{subfigure}[b]{5cm} \label{somme_unif}
      \includegraphics[width=\textwidth]{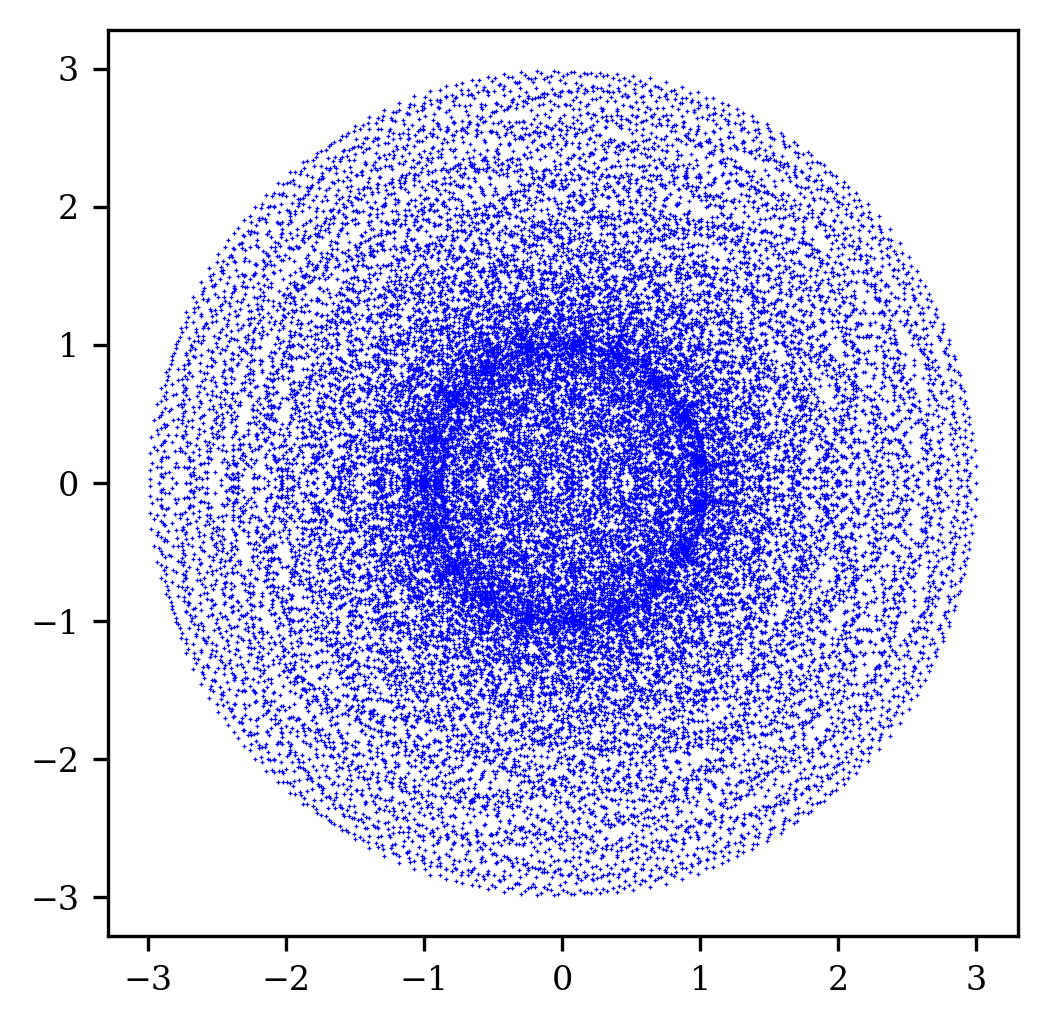}
      \caption{$g = X^3 + 2X^2 + 3$ and \\
        $q=30113$.}
    \end{subfigure}
    \hspace{0.6cm}
    \begin{subfigure}[b]{4.42cm}
      \includegraphics[width=\textwidth]{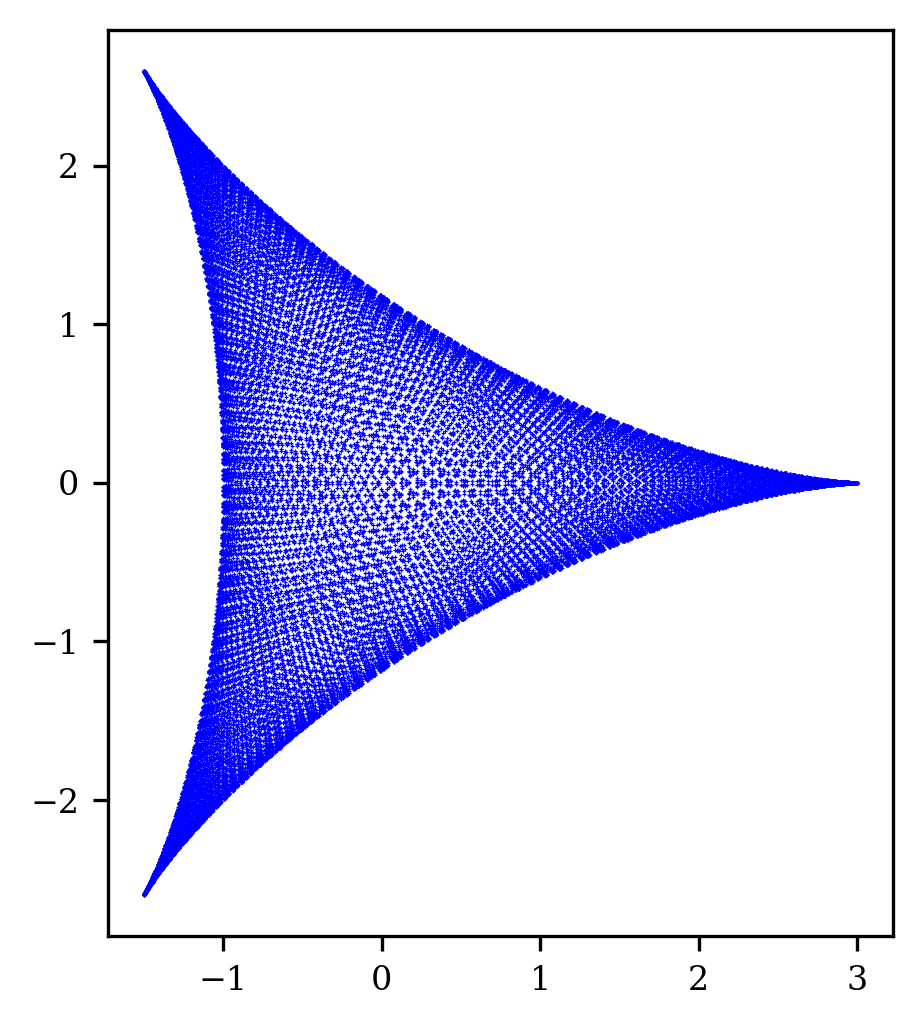}
      \caption{$g = X^3 + X + 3$ and $q=30223$.}
    \end{subfigure}
    \caption{The sums $\sum_{x \in Z_g(\Ff_q)} e( \frac{ax}{q})$
      as $a$ varies in $\Ff_q$, for two different polynomials
      $g$ of degree $3$.}
    \label{relations_or_not}
  \end{figure}
  \par
  (3) The second part of the theorem also has precursors: for instance,
  the result follows from~\cite[Prop.\,3.2]{clt} if
  $g=(X-1)\cdots (X-d)$. We illustrate our generalization in Figure \ref{sum_of_3_ST} with the example of another polynomial $g$ of degree 3.
  
  \begin{figure}
    \centering \includegraphics[width =
    0.6\textwidth]{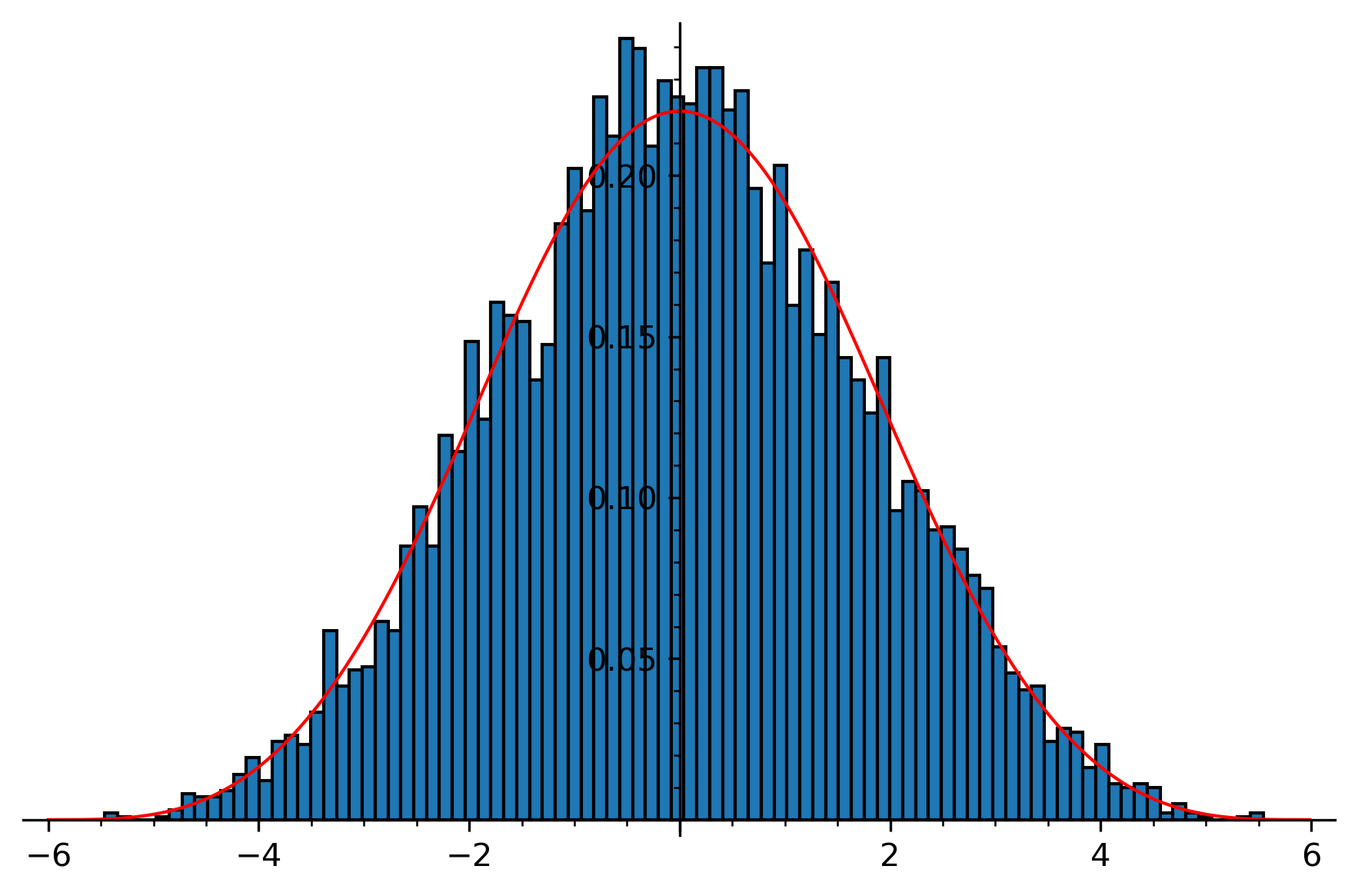}
    \caption{Distribution of the values of the sums
      $\sum_{x \in Z_g(\Ff_q)} \mathrm{Kl}_2(ax;q)$ as $a$ varies
      in $\Ff_q$, for $g = X^3 -9X -1$ and $q = 8089$. The red
      curve is the probability density function of the random
      variable $X_1 + X_2 + X_3$ defined as the sum of three
      independent and identically distributed Sato--Tate random
      variables.}
    \label{sum_of_3_ST}
  \end{figure}
\end{example}

\begin{remark}
  We assume that~$g$ is monic mostly for simplicity to ensure that the
  roots of~$g$ are algebraic integers. However, one can also handle an
  arbitrary polynomial~$g$ by considering an integer~$N\geq 1$ such
  that $Nz$ is integral for all roots~$z$ of~$g$, and either reduce to
  the monic case by using the polynomial~$\widetilde{g}$ with roots
  the~$Nz$, or by considering below the ring~$\Oo_g[1/N]$
  of~$N$-integers in~$K_g$ instead of the full ring of integers.
\end{remark}

\subsection*{Notation}
\label{section notations}

Let~$G$ be a locally compact abelian group, with character
group~$\widehat{G}$. Let~$H$ be a closed subgroup of~$G$. We recall that
the restriction homomorphism from $\widehat{G}$ to $\widehat{H}$ is
surjective (in other words, any character of $H$ can be extended to a
character of $G$).

The \emph{orthogonal} $H^{\perp}$ of~$H$ is the closed subgroup
of~$\widehat{G}$ defined by
$$
H^{\perp}=\{\chi \in \widehat{G}\,\mid\, \chi(x)=1\text{ for all } x \in
H\}.
$$

If we identify the dual of~$\widehat{G}$ with~$G$ by Pontryagin duality,
then the orthogonal of~$H^{\perp}$ is identified with~$H$, or in other
words
$$
H=\{x \in G\,\mid\, \chi(x)=1\text{ for all } \chi \in
H^{\perp}\}.
$$

We refer, e.g., to Bourbaki's account~\cite{ts2} of Pontryagin duality
for these facts.

Suppose that~$G$ is compact. A random variable with values in~$G$ is
said to be \emph{uniformly distributed on~$H$} if its law~$\mu$ is the
probability Haar measure on~$H$ (viewed as a probability measure
on~$G$).

Throughout this paper, we will consider a fixed \emph{monic}
polynomial $g\in\Zz[X]$ of degree~$d\geq 1$. We denote by $Z_g$ the
set of zeros of~$g$ in~$\Cc$, and more generally by $Z_g(A)$ the set
of zeros of~$g$ in any commutative ring~$A$. We further denote by
$K_g$ the splitting field of~$g$ in~$\Cc$, so that
$K_g=\Qq(Z_g)$. Since our discussion only depends on $Z_g$, we will
assume, without loss of generality, that $g$ is separable. Since~$g$
is monic, the set~$Z_g$ is contained in the ring of integers~$\Oo_g$
of~$K_g$.

For any set~$X$, we denote by $C(Z_g;X)$ the set of functions
$Z_g\to X$, and in particular we write $C(Z_g)=C(Z_g;\Cc)$ for the
vector space of $\Cc$-valued functions on~$Z_g$. We denote by $\sigma$
the linear form on~$C(Z_g)$ defined by
$$
\sigma(f)=\sum_{x\in Z_g}f(x),
$$
and by $\gamma$ the morphism of abelian groups from $C(Z_g;\Zz)$ to~$K_g$
defined by
\begin{equation}\label{eq-gamma}
  \gamma(f)=\sum_{x\in Z_g}f(x)x.
\end{equation}

The set $C(Z_g;\Ss^1)$ is a compact abelian group, isomorphic to
$(\Ss^1)^{|Z_g|}$ by sending $\alpha$ to $(\alpha(z))_{z\in Z_g}$.

\subsection*{Acknowledgements} The second author wishes to thank
Florent Jouve and Guillaume Ricotta for many helpful discussions, and
Emanuele Tron for giving the key ideas of the proof of the linear
independence of $j$-invariants. Pictures were made using the
open-source software \texttt{sagemath}.

\section{The case of additive characters}\label{sec-proofs1}

We begin with the simple setup of the first part of Theorem~\ref{th-1}
before considering a much more general situation.

We fix a separable monic polynomial~$g\in\Zz[X]$ as in the previous
discussion.  Let $p$ be a non-zero prime ideal in $\Oo_g$. We denote
by $|p|=|\Oo_g/p|$ the norm of~$p$. For any ideal $I\subset \Oo_g$,
the canonical projection $\Oo_g\to \Oo_g/I$ will be denoted
$\varpi_I$, or simply $\varpi$ when the ideal is clear from context.

We denote by $\mathcal{S}_g$ the set of prime ideals $p\subset \Oo_g$
which do not divide the discriminant of~$g$ (so that the reduction map
modulo~$p$ is injective on~$Z_g$) and have residual degree one (in
particular, these are unramified primes). For $p\in\mathcal{S}_g$, the
norm $q=|p|$ is a prime number, and for any integer $n\geq 1$, the
restriction $\Zz\to \Oo_g/p^n$ of~$\varpi_{p^n}$ induces a ring
isomorphism $\Zz/q^n\Zz\to \Oo_g/p^n$. We will usually identify these
two rings. Moreover, for $p\in\mathcal{S}_g$ and $n\geq 1$, the
separable polynomial $g$ has $\deg(g)$ different roots in the completion
of~$K_g$ at~$p$, hence the reduction map modulo $p^n$ induces a
bijection $Z_g\to Z_g(\Oo_g/p^n)=Z_g(\Zz/|p|^n\Zz)$ for any integer
$n\geq 1$.

For any prime ideal $p\in\mathcal{S}_g$ and any integer~$n\geq 1$, we
view $\Oo_g/p^n$ as a finite probability space with the uniform
probability measure.  We define random variables $U_{p^n}$
on~$\Oo_g/p^n$, taking values in $C(Z_g;\Ss^1)$, by
$$
U_{p^n}(a)(x)=e\Bigl(\frac{a\varpi(x)}{|p|^n}\Bigr),
$$
where $\varpi=\varpi_{p^n}$ here (according to our convention,
$a\varpi(x)$ is an element of $\Oo_g/p^n$ which is identified to an
element of $\Zz/|p|^n\Zz$).

\begin{remark}
  In the earlier references \cite{dgl}, \cite{ghl} and \cite{untrau},
  we have $g=X^d-1$ for some
  integer~$d$, and one considers primes $q\equiv
  1\mods{d}$. A primitive $d$-th root of unity modulo~$q$, say
  $w_q$, is fixed for all
  such~$q$, and one considers the limit
  as~$q\to~\infty$ of the tuples $ (e(\frac{aw_q^k}{q}))_{0 \leqslant
    k \leqslant d-1}$, for $a$ uniform
  in~$\Ff_q$.  This approach does not generalize in a convenient way
  to more general
  polynomials~$g$, where the roots are not as easily parameterized.
\end{remark}

\begin{proposition}[Ultra-short equidistribution]\label{pr-1}
  The random variables $U_{p^n}$ converge in law as $|p|^n \to+\infty$
  to a random function $U\colon Z_g\to \Ss^1$ such that~$U$ is
  uniformly distributed on the subgroup $H_g\subset C(Z_g;\Ss^1)$ which
  is orthogonal to the abelian group
  $$
  R_g=\ker(\gamma)= \{ \alpha\in C(Z_g;\Zz)\,\mid\, \sum_{x\in
    Z_g}\alpha(x)x=0 \}
  $$
  of (integral) additive relations betweens roots of~$g$, i.e.
  $$
  H_g=\{f\in C(Z_g;\Ss^1)\,\mid\, \text{ for all $\alpha\in R_g$, we
    have } \prod_{x\in Z_g}f(x)^{\alpha(x)}=1\}.
  $$
\end{proposition}

\begin{proof}
  Since $C(Z_g;\Ss^1)$ is a compact abelian group, we can apply the
  generalized Weyl Criterion for equidistribution: it is enough to
  check that, for any character $\eta$ of $C(Z_g;\Ss^1)$, we have
  $$
  \expect(\eta(U_{p^n}))\to \expect(\eta(U))
  $$
  as $|p|^n\to +\infty$. The right-hand side is either~$1$ or~$0$,
  depending on whether the restriction of~$\eta$ to~$H_g$ is trivial or
  not.

  The character~$\eta$ is determined uniquely by a function
  $\alpha\in C(Z_g;\Zz)$ by the rule
  $$
  \eta(f)=\prod_{x\in Z_g}f(x)^{\alpha(x)}
  $$
  for any $f\in C(Z_g;\Ss^1)$.  We have then by definition
  $$
  \expect(\eta(U_{p^n})) = \frac{1}{|p|^n} \sum_{a\in \Oo_g/p^n}
  e\Bigl(\frac{a}{|p|^n} \varpi\Bigr( \sum_{x\in Z_g}\alpha(x)x\Bigr)
  \Bigr).
  $$

  Simply by orthogonality of the characters modulo $|p|^n$,
  this sum is either~$1$ or~$0$, depending on whether
  $$
  \gamma(\alpha)=\sum_{x\in Z_g}\alpha(x)x
  $$
  is zero modulo $p^n$ or not. As soon as $|p|^n$ is large enough,
  this condition is equivalent with $\gamma(\alpha)$ being zero or not
  in~$K_g$. In particular, the limit of~$\expect(\eta(U_{p^n}))$ is
  either~$1$ or~$0$ depending on whether $\alpha\in \ker(\gamma)=R_g$
  or not, and this is exactly what we wanted to prove.
\end{proof}

\begin{remark}
  The proof shows that in fact the Weyl sums are \emph{stationary}. This
  somewhat unusual feature\footnote{\ Though there are important
    instances of limit theorems where \emph{moments} are stationary,
    e.g. in the convergence of the number of fixed points of random
    permutations to a Poisson distribution.} explains the very regular
aspect of the experimental pictures. We will explore further
consequences of this fact in a later work.
\end{remark}

\begin{corollary} \label{cor-additive-characters} For $a$ taken
  uniformly at random in $\Oo_g/p^n$ with $p\in\mathcal{S}_g$, lying
  above a prime number $q$ which does not divide $\disc(g)$, the sums
  $$
  \sum_{x\in Z_g(\Oo_g/ p^n)}e\Bigl(\frac{ax}{|p|^n}\Bigr)
  $$
  become equidistributed in~$\Cc$ as $|p|^n\to +\infty$ with limiting
  measure~$\mu_g$ given by the law of $\sigma(U)$, where $U$ is
  uniformly distributed on~$H_g$.
  \par
  Similarly, for a prime number $q$ totally split in~$K_g$ and
  not dividing the discriminant of~$g$, the sums
  $$
  \sum_{\substack{x\in \Zz / q^n \Zz\\g(x) \equiv 0
      \mods{q^n}}}e\Bigl(\frac{ax}{q^n}\Bigr)
  $$
  for $a\in \Zz / q^n \Zz$ become equidistributed in~$\Cc$ as
  $q^n \to+\infty$ with limit $\sigma(U)$.
\end{corollary}

\begin{proof}
  Since $\varpi_{p^n}$ induces a bijection between $Z_g$ and
  $Z_g(\Oo_g/p^n)$, the random variables whose limit we are
  considering coincide with $\sigma(U_{p^n})$, and since $\sigma$ is a
  continuous function from $C(Z_g;\Ss^1)$ to~$\Cc$, we obtain the
  result from Proposition~\ref{pr-1} by composition.
  \par
  For the second part, we note that for any prime number $q$ which is
  totally split in~$\Oo_g$ and does not divide the discriminant of
  $g$, there exists a prime ideal $p \in~\mathcal{S}_g$ above~$q$, and
  for any $n\geq 1$, we have then $Z_g(\Zz/q^n\Zz)=Z_g(\Oo_g/p^n)$, so
  that
  $$
  \sum_{x\in Z_g(\Oo_g/ p^n)}e\Bigl(\frac{ax}{|p|^n}\Bigr) =
  \sum_{\substack{x\in \Zz / q^n \Zz\\g(x) \equiv 0
      \mods{q^n}}}e\Bigl(\frac{ax}{q^n}\Bigr),
  $$
  and the result follows from the first part since we are considering
  a subsequence of the random variables previously considered.
\end{proof}



Before studying a few examples in the next section, we make a few
remarks concerning the limiting measures. Since the random variable
$\sigma(U)$ is bounded, one can compute all its moments using the
equidistribution.  This leads straightforwardly to the formulas
$$
\expect(\sigma(U))=\begin{cases}
  0&\text{ if $0\notin Z_g$}\\
  1&\text{ if $0\in Z_g$},
\end{cases}
$$
and
$$
\expect(|\sigma(U)|^2)=|Z_g|.
$$




The fact that the expectation is zero if $g$ is irreducible of degree at
least~$2$ has some indirect relevance to the well-known conjecture
according to which the fractional parts of the roots modulo primes
$q\leq x$ of an irreducible polynomial~$g$ of degree at least~$2$ should
become equidistributed (with respect to the Lebesgue measure)
in~$\Rr/\Zz$ as $x\to +\infty$ -- see, e.g., the paper~\cite{dfi} of
Duke, Friedlander and Iwaniec.
\par
Indeed, the Weyl sums for this equidistribution problem are
(essentially)
$$
\frac{1}{\pi(x)}\sum_{q\leq x}\sum_{\substack{y\in
    \Ff_q\\g(y)=0}}e\Bigl(\frac{ay}{q}\Bigr)
$$
(where~$q$ ranges over primes) for some \emph{fixed} non-zero
integer~$a$.  For each prime~$q$ which happens to be totally split
in~$K_g$, the inner sum is of the form $\sigma(U_p(a))$ for some prime
ideal~$p\in\mathcal{S}_g$. Thus, Proposition~\ref{pr-1} tells us about
the asymptotic distribution of these terms \emph{when $a$ varies
  modulo~$q$}. Intuitively, we may hope that the average over~$q$
should lead to a limit which coincides with $\expect(\sigma(U))=0$,
and this would translate to the equidistribution conjecture.
\par
In fact, we may even ask whether these inner parts of the Weyl sums
for equidistribution are \emph{themselves} equidistributed. More
precisely, fix a non-zero integer~$a$, and consider the random
variables of the type
$$
U'_T(p)(x)=e\Bigl(\frac{a\varpi_p(x)}{|p|}\Bigr)
$$
defined on the probability spaces $\mathcal{S}_g(T)$ of primes~$p$ in
$\mathcal{S}_g$ with $|p|\leq T$ (with uniform probability
measure), and with values in~$C(Z_g;\Ss^1)$.

\begin{question}
  Do the random functions $U'_T$ converge in law as $T\to+\infty$? If
  Yes, is the limit the same as in Proposition~\ref{pr-1}?
\end{question}

If the answer to this question is positive, then the equidistribution
conjecture holds, at least when averaging only over primes totally
split in~$K_g$, since  then
$$
\frac{1}{|\mathcal{S}_g(T)|}\sum_{p\in\mathcal{S}_g(T)} \sum_{x\in
  Z_g}e\Bigl(\frac{a\varpi_p(x)}{|p|}\Bigr) \to
\int_{\Cc}zd\mu_g(z)=0.
$$

The answer is indeed positive when $g$ is irreducible of degree~$2$,
by the work of Duke, Friedlander and Iwaniec~\cite{dfi} and
Toth~\cite{toth} (more precisely, in this case the relevant inner Weyl
sums are essentially Salié sums, and it is proved -- using the
equidistribution property for the roots of quadratic congruences,
which is the main result of these papers -- that the Salié sums become
equidistributed in $[-2,2]$ like the sums $e(x)+e(-x)$ where~$x$ is
uniformly distributed in~$\Rr/\Zz$. Moreover, this question is closely
related with recent conjectures of
Hrushovski~\cite[\S\,5.5]{hrushovski}, themselves motivated by
questions concerning the model theory of finite fields with an
additive character.
\par
Numerical experiments also seem to suggest a positive answer at least
in many cases. But note also that obtaining the same limiting measure
depends on assuming that~$g$ is irreducible. (For instance, if there
is an integral root~$k$ for~$g$, as is the case with $k=1$ for
$X^d-1$, then the value $U'_T(p)(k)=e(ak/|p|)$ converges to~$1$ as
$|p|\to +\infty$, which is a different behavior than that provided by
Proposition~\ref{pr-1}.)

\section{Examples}\label{sec-examples}

We now consider a few examples of Proposition~\ref{pr-1}.
\par
(1) Suppose that $g=X^d-1$ for some $d\geq 1$, so that $Z_g=\mmu_d$ is
the group of $d$-th roots of unity.
\par
Consider first the case when $d=\ell$ is a prime number. The group of
additive relations is generated in this case by the constant function
$\alpha=1$ (indeed, let $\xi\in\mmu_{\ell}$ be a root of unity different
from~$1$; then a relation
$$
\sum_{x\in\mmu_\ell}\alpha(x)x=0
$$
is equivalent to $f(\xi)=0$, where $f$ is the polynomial
$$
\sum_{i=0}^{\ell-1}\alpha(\xi^i)X^i\in\Zz[X],
$$
which must therefore be an integral multiple of the minimal polynomial
$$
1+X+\cdots +X^{\ell-1}
$$
of $\xi$). The subgroup $H_{X^{\ell}-1}$ which is the  support of the
limit~$U$ in this case is then 
$$
H_{X^{\ell}-1}=\{f\colon \mmu_\ell \to\Ss^1\,\mid\,
\prod_{x\in\mmu_{\ell}}f(x)=1\}.
$$
which can be identified with~$(\Ss^1)^{\ell-1}$ by the group
isomorphism $f\mapsto (f(x))_{x\in\mmu_{\ell}\setminus\{1\}}$. The linear
form  $\sigma$ is then identified with the linear form
$(\Ss^1)^{\ell-1}\to \Cc$ such that
$$
(y_1,\ldots,y_{\ell-1})\mapsto y_1+\cdots
+y_{\ell-1}+\frac{1}{y_1\cdots y_{\ell-1}}.
$$
\par
In the case of a general~$d$, the same argument shows that
$R_{X^{d}-1}$ is the group of functions $\alpha\colon\mmu_d\to\Zz$
such that the $d$-th cyclotomic polynomial $\Phi_d$ divides
$$
\sum_{i=0}^{d-1}\alpha(\xi^i)X^i,
$$
where $\xi$ is a primitive $d$-th root of unity. Thus $R_{X^d-1}$ is a
free abelian group of rank $d-\varphi(d)$, generated by the
functions~$\alpha$ corresponding to the polynomials
$$
\Phi_d,\ X\Phi_d,\ \cdots,\ X^{d-\varphi(d)-1}\Phi_d.
$$
\par
Although this presentation is more abstract, it coincides with the
description of Duke, Garcia and Lutz in~\cite[Th.\,6.3]{dgl}.
\par
(2) The group of additive relations of a polynomial is studied by
Berry, Dubickas, Elkies, Poonen and Smyth~\cite{rels} in some detail
(see also~\cite{relations}). It is known for instance (see
e.g.~\cite[Prop.\,2.8]{relations} or~\cite[Prop.\,4.7.12]{repr}; this
goes back at least to Smyth~\cite{smyth}) that if the Galois group of
$K_g$ over~$\Qq$ is the symmetric group~$\mathfrak S_d$, then only two
cases are possible: either $R_g$ is trivial (in which case the limit
measure $\mu_g$ is the law of the sum of $d$ independent random
variables uniformly distributed on~$\Ss^1$) or $R_g$ is generated by
the constant function~$1$ (in which case the measure $\mu_g$ is the
same measure described in (1), except that $d$ is not necessarily
prime here). This second case corresponds to the situation where the
sum of the roots is zero, i.e., to the case when the coefficient of
$X^{d-1}$ in~$g$ is zero.
\par
  
(3) More interesting examples arise from polynomials $g$ that are
characteristic polynomials of ``random'' elements of the group of
integral matrices in a simple Lie algebra~$L$, where additive relations
corresponding to the root system of~$L$ will appear. For instance, for
the Lie algebra of type $G_2$, in its $7$-dimensional irreducible
representation, the roots of a characteristic polynomial have the form
of tuples
$$
(0,x,y,x+y,-x,-y,-x-y)
$$
so that the group of additive relations will be quite large. It would
be interesting to determine explicitly the support of the image
measure in this case.
\par
(4) Another natural example comes from the \emph{Hilbert class
  polynomial} $g=H_{\Delta}$, whose roots are the $j$-invariants of
elliptic curves with CM by an imaginary quadratic order~$\Oc$ of given
discriminant $\Delta$ (see, e.g.,~\cite[\S\,13,\,
Prop.\,13.2]{cox}). This means that we consider sums
\begin{equation} \label{sum-elliptic}
  \sum_{E\text{ with CM by } \Oc}e\Bigl(\frac{aj(E)}{q}\Bigr),
\end{equation}
where the sum runs over isomorphism classes over~$\Cc$ of elliptic
curves with CM by~$\Oc$, for prime numbers~$q$ totally split in the ring
class field corresponding to the order~$\Oc$. For instance, if
$\Delta=-4m$ with $m\geq 1$ squarefree, these are exactly the primes of
the form $x^2+my^2$ (see the book of Cox~\cite{cox} for details).

From Proposition \ref{pr-1}, and
Corollary~\ref{cor-additive-characters}, we know that the asymptotic
distribution of the sums \eqref{sum-elliptic}, as $q$ tends to
infinity and $a$ varies in $\Ff_q$, is governed by the additive
relations between these $j$-invariants. As it turns out, there are no
non-trivial relations, except for~$\Delta=-3$. This is essentially due
to the fact that there is one $j$-invariant (for fixed $\Oc$ with
discriminant large enough) which is much larger than the others,
combined with the following lemma.

\begin{lemma}
  Let~$g\in \Zz[X]$ be irreducible over~$\Qq$ of degree~$d\geq 2$. If
  there exists $x_0\in Z_g$ such that
  $$
  |x_0|>\sum_{\substack{x\in Z_g\\x\not=x_0}}|x|,
  $$
  then $R_g=\{0\}$.
\end{lemma}

\begin{proof}
  Suppose that there exists $\alpha\in R_g$ non-zero. Let $x_1\in Z_g$ be
  such that $|\alpha(x_1)|$ is maximal, hence non-zero. Dividing
  by~$\alpha(x_1)$, we obtain a relation
  $$
  0=\sum_{x\in Z_g}\beta(x) x
  $$
  where $\beta(x)\in\Qq$ with $|\beta(x)|\leq 1$ for all~$x$ and
  $\beta(x_1)=1$. Since~$g$ is irreducible, we can find a Galois
  automorphism $\xi$ such that $\xi(x_1)=x_0$, which means that we may
  assume that~$x_1=x_0$. Then we obtain
  $$
  |x_0|=\Bigl|\sum_{x\not=x_0}\beta(x)x\Bigr|\leq \sum_{x\not=x_1}|x|,
  $$
  and we conclude by contraposition.
\end{proof}

This lemma is applicable to the Hilbert class polynomial
$H_{\Delta}$. Indeed, it is irreducible (see,
e.g.,~\cite[\S\,13]{cox}).  To check the existence of a dominating
$j$-invariant, we use the bound
$$ 
\Bigl| |j(\tau)| - e^{2\pi \Imag(\tau)}\Bigr| \leqslant 2079,
$$
for $\tau$ in the usual fundamental domain~$F$ of $\Hh$ modulo
$\SL_2(\Zz)$ (see \cite[Lemma\,1]{bilu-masser-zannier} by Bilu, Masser
and Zannier), combined with the fact that there is a unique $\tau$
in~$F$ such that $j(\tau)$ is a root of $H_{\Delta}$ and
$\Imag(\tau)\geq \sqrt{|\Delta|}/2$, while all other $j$-invariants
for the order $\Oc$ are of the form $j(\tau')$ where $\tau'\in F$ has
$\Imag(\tau')\leq \sqrt{|\Delta|}/4$
(see~\cite[Section\,3.3]{allombert} by Allombert, Bilu and
Pizarro-Madariaga).  These properties imply that the lemma is
applicable as soon as the bound
$$
e^{\pi \sqrt{|\Delta|}} - 2079>\deg(H_{\Delta})(e^{\pi
  \frac{\sqrt{|\Delta|}}{2}} + 2079)
$$
holds. The degree of~$H_{\Delta}$ is the Hurwitz class number, and one
knows classically that
$$
\deg H_{\Delta}\leq \frac{\sqrt{|\Delta|}}{\pi}(\log |\Delta|+2),
$$
(see, e.g.,~\cite[Lemma\,3.6]{b-h-k} by Bilu, Habegger and Kühne). One
checks easily that the desired bound follows unless $\Delta\geq
-9$. For the remaining cases, $H_{\Delta}$ has degree~$1$, and its
unique root is a non-zero integer, except that~$H_{-3}=X$ (see for
instance the table~\cite[\S\,12.C]{cox} in the book of
Cox). Therefore, unless $\Delta = -3$, the module of additive
relations of $H_{\Delta}$ is trivial. Of course, for~$\Delta=-3$, it
is isomorphic to~$\Zz$.

This immediately leads to the following corollary concerning the distribution of sums of
type \eqref{sum-elliptic}:

\begin{corollary}
  Fix a negative discriminant $\Delta \neq -3$ of an imaginary
  quadratic order~$\Oc$ with class number~$h$. As $q \to \infty$ among
  the primes totally split in the ring class field corresponding to
  the order $\Oc$, the sums
  $$
  \sum_{E\text{ \rm with CM by } \Oc}e\Bigl(\frac{aj(E)}{q}\Bigr)
  $$
  parametrized by $a \in \Ff_q$ become equidistributed in $\Cc$ with
  respect to the law of the sum $X_1+\cdots+X_{h}$ of $h$ independent
  random variables, each uniformly distributed on the unit circle.
\end{corollary}

On the other hand, for~$\Delta=-3$, we have
$$
\sum_{E\text{ \rm with CM by } \Oc}e\Bigl(\frac{aj(E)}{q}\Bigr)=1
$$
for all~$q$.


\section{Conditioning}

The basic argument leading to Proposition~\ref{pr-1} extends in another
nice way to the \emph{conditioning} situation, where we restrict the
random variables $U_{p^n}$ to suitable subsets of~$\Oo_g/p^n$. This
turns out to be closely related to the distribution of the fractional
parts of these subsets.

The precise statements require some additional notation. First, we
define by $\ind(g)$ the non-negative integer such that
$$
\Imag(\gamma)\cap \Zz=\ind(g)\Zz
$$
(recall the definition~(\ref{eq-gamma}) of~$\gamma$; note that it is
possible that~$\kappa=0$, e.g. for $g=X^2+d$ with $d\not=0$).

For a prime ideal $p\in\mathcal{S}_g$ and $n\geq 1$, and for any
$a\in\Oo_g/p^n$, we define the ``fractional part'' of~$a$ to be the
fractional part in $[0,1]$ of $\bar{a}/|p|^n$ for any lift
$\bar{a}\in\Zz$ of $a$ identified as an element of~$\Zz/|p|^n\Zz$.

We denote by $U$ the limit in Proposition~\ref{pr-1}.

\begin{proposition}[Ultra-short equidistribution]\label{pr-cond}
  For a subsequence of ideals $p^n$ with $p\in\mathcal{S}_g$ and
  $n\geq 1$, let $A_{p^n}$ be a non-empty subset of~$\Oo_g/p^n$. 
  \par
  \emph{(1)} If the fractional parts of $a\in A_{p^n}$ are
  \emph{uniformly equidistributed modulo~$1$} as $|p|^n\to +\infty$,
  in the sense that
  $$
  \max_{\substack{h\in \Oo_g/p^n\\h\not=0}}
  \frac{1}{|A_{p^n}|}\Bigl|\sum_{a\in A_{p^n}}
  e\Bigl(\frac{ah}{|p|^n}\Bigr)\Bigr|\to 0
  $$
  as $|p|^n\to +\infty$, then the restriction of the random variables
  $U_{p^n}$ to~$A_{p^n}$, viewed as probability space with uniform
  probability measure, converge in law to~$U$.
  \par
  \emph{(2)} Suppose that~$\ind(g)\not=0$ and that the restriction of
  the random variables $U_{p^n}$ to~$A_{p^n}$, viewed as probability
  space with uniform probability measure, converge in law to~$U$. Then
  the fractional parts of elements of $\ind(g)A_{p^n}$ are
  equidistributed modulo~$1$.
\end{proposition}

\begin{proof}
  We denote by $U'_{p^n}$ the restriction of~$U_{p^n}$ to~$A_{p^n}$,
  viewed as probability space with the uniform probability measure.

  We expand the characteristic function $f_{p^n}$ of $A_{p^n}$ in
  discrete Fourier series
  $$
  f_{p^n}(a)=\sum_{h\in
    \Oo_g/p^n}\alpha_{p^n}(h)e\Bigl(\frac{ha}{|p|^n}\Bigr)
  $$
  where
  $$
  \alpha_{p^n}(h)=\frac{1}{|p|^n}
  \sum_{a\in A_{p^n}}e\Bigl(-\frac{ha}{|p|^n}\Bigr).
  $$

  Let~$\eta$ be a character of $C(Z_g;\Ss^1)$, determined by
  $\alpha\in C(Z_g;\Zz)$ as in Proposition~\ref{pr-1}. By definition, we
  have
  \begin{align*}
    \expect(\eta(U'_{p^n}))&=\frac{1}{|A_{p^n}|} \sum_{a\in A_{p^n}}
    e\Bigl(\frac{a}{|p|^n} \varpi\Bigr( \sum_{x\in Z_g}\alpha(x)x\Bigr)
    \Bigr)
    \\
    &= \frac{1}{|A_{p^n}|} \sum_{h\in \Oo_g/p^n}\alpha_{p^n}(h)
    \sum_{a\in \Oo_g/p^n} e\Bigl(\frac{a}{|p|^n}(
    \varpi(\gamma(\alpha))+h) \Bigr)
    \\
    &=\frac{|p|^n}{|A_{p^n}|} \alpha_{p^n}(-\varpi(\gamma(\alpha)))=
    \frac{1}{|A_{p^n}|} \sum_{a\in
      A_{p^n}}e\Bigl(\frac{\varpi(\gamma(\alpha))a}{|p|^n}\Bigr),
  \end{align*}
  an identity between Weyl sums for the equidistribution of~$U_{p^n}$
  and Weyl sums for the equidistribution of the fractional parts of
  elements of~$A_{p^n}$.

  Suppose first that $A_{p^n}$ is uniformly equidistributed modulo~$1$.
  If $\gamma(\alpha)=0$, then we get
  $\expect(\eta(U'_{p^n}))=1$. Otherwise, for $|p|^n$ large enough, we
  get $\varpi(\gamma(\alpha))\not=0\in \Oo_g/p^n$, and therefore
  $$
  |\expect(\eta(U'_{p^n}))|\leq \max_{\substack{h\in
    \Oo_g/p^n \\ h \neq 0}}\frac{1}{|A_{p^n}|}\Bigl|\sum_{a\in A_{p^n}}
  e\Bigl(\frac{ah}{|p|^n}\Bigr)\Bigr|,
  $$
  which tends to~$0$ by assumption. This proves the first statement.

  Conversely, suppose that $\ind(g)\not=0$ and that $U'_{p^n}$ converges
  in law to~$U$. Let~$h\in\Zz\setminus \{0\}$. Pick
  $\alpha\in C(Z_g;\Zz)$ such that $\gamma(\alpha)=\ind(g)h$, which
  exists by definition of~$\ind(g)$. For all $p^n$, we get
  $$
  \frac{1}{|A_{p^n}|} \sum_{a\in
    A_{p^n}}e\Bigl(\frac{h\ind(g)a}{|p|^n}\Bigr)= \frac{1}{|A_{p^n}|}
  \sum_{a\in A_{p^n}}e\Bigl(\frac{\varpi(\gamma(\alpha))a}{|p|^n}\Bigr)=
  \expect(\eta(U'_{p^n}))
  $$
  where $\eta$ is the character of $C(Z_g;\Ss^1)$ corresponding
  to~$\alpha$. This character is not trivial on~$H_g$ (because
  $\gamma(\alpha)\not=0$), and therefore
  $$
  \lim_{|p|^n\to +\infty} \frac{1}{|A_{p^n}|} \sum_{a\in
    A_{p^n}}e\Bigl(\frac{h\ind(g)a}{|p|^n}\Bigr)=0,
  $$
  which proves equidistribution modulo~$1$ of fractional parts of
  $\ind(g)A_{p^n}$ by the Weyl Criterion.
\end{proof}

\begin{example}
  (1) Let $\alpha\in \Rr$ satisfy $0<\alpha<1$. Let $A_{p^n}$ be the
  set of classes corresponding to an interval of length
  $\sim \alpha |p|^n$ in $\Zz/|p|^n\Zz$. Then equidistribution (and a
  fortiori uniform equidistribution) of the fractional parts
  \emph{fails}, hence the second part implies, by contraposition, that
  if $\ind(g)=1$, then the random variables $U_{p^n}$ conditioned to
  have $a\in A_{p^n}$ \emph{do not} converge to~$U$.
  \par
  As an illustration, let $g= X^3 + X^2 + 2X + 1$. One checks quickly
  that~$g$ is irreducible, with Galois group $\mathfrak{S}_3$, so that
  Example~2 of Section \ref{sec-examples} implies that the sums
  \begin{equation} \label{sum_interval}
    \sum_{x \in Z_g(\Ff_q)}^{} e\Bigl( \frac{ax}{q} \Bigr),
  \end{equation}
  parametrized by $a \in \Ff_q$ for $q$ totally split in $K_g$, become
  equidistributed with respect to the measure $\mu_g$ which is the law
  of the sum of three independent random variables, each uniformly
  distributed on $\Ss^1$. A plot of the values
  $ \sum_{x \in Z_g(\Ff_q)} e(\tfrac{ax}{q})$ for $a \in \Ff_q$ would
  then be very similar to Figure~\ref{relations_or_not}, (A).

  However, this polynomial $g$ satisfies $\kappa(g) = 1$ (since the
  coefficient~$1$ of~$X^2$ shows that the sum of the roots, which is
  an element of~$\Imag(\gamma)\cap \Zz$, is $-1$) and hence these
  sums, parametrized by $a \in \{ 0, \dots, \frac{q-1}{2}\}$, do not
  become equidistributed with respect to the same measure. Numerical
  experiments confirm this (see Figure~\ref{different_limit}), but
  suggest that there is equidistribution with respect to another
  measure.
  
  \begin{figure}
    \centering
    \begin{subfigure}[b]{0.32\textwidth}
      \centering
      \includegraphics[width=\textwidth]{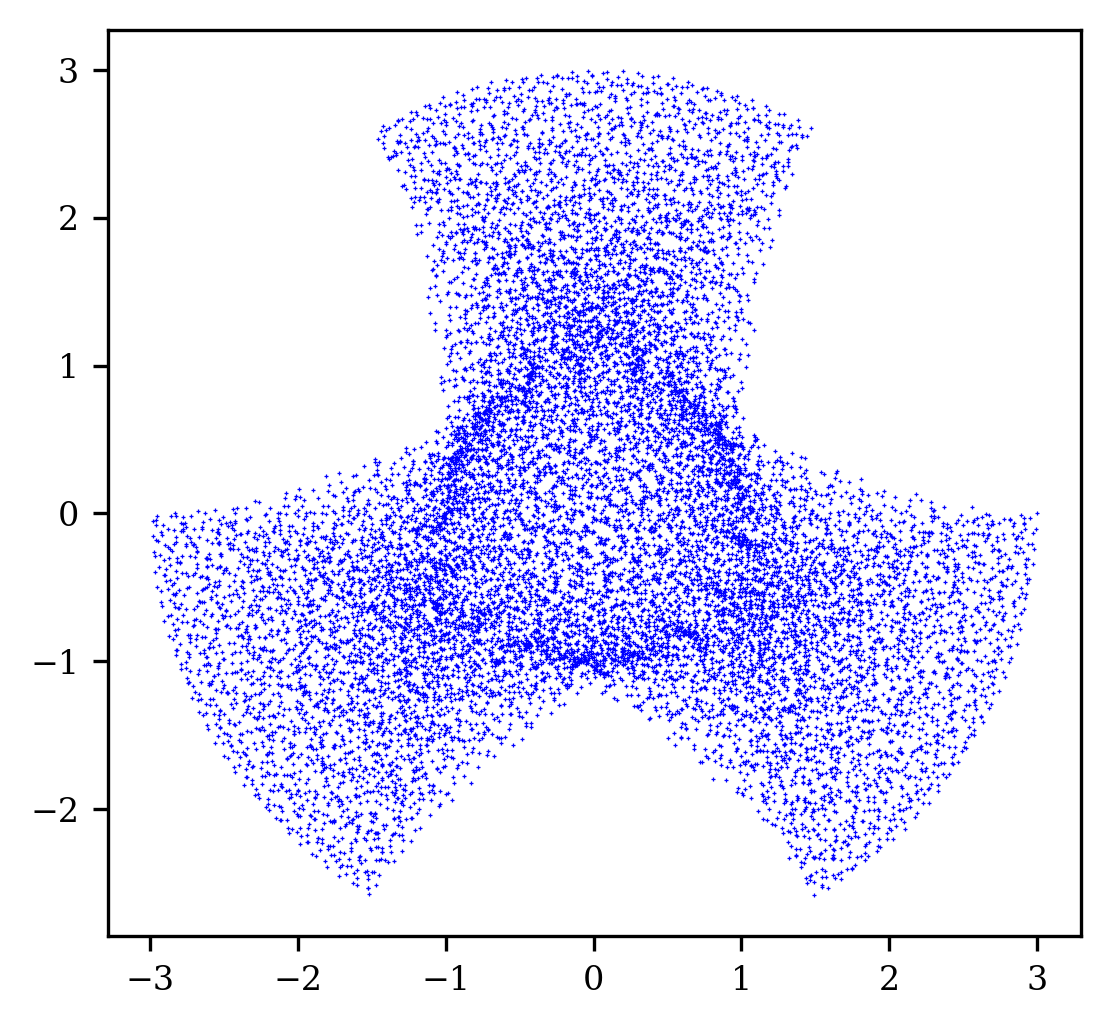}
      \caption{$q = 30307$}
    \end{subfigure}
    \hfill
    \begin{subfigure}[b]{0.32\textwidth}
      \centering
      \includegraphics[width=\textwidth]{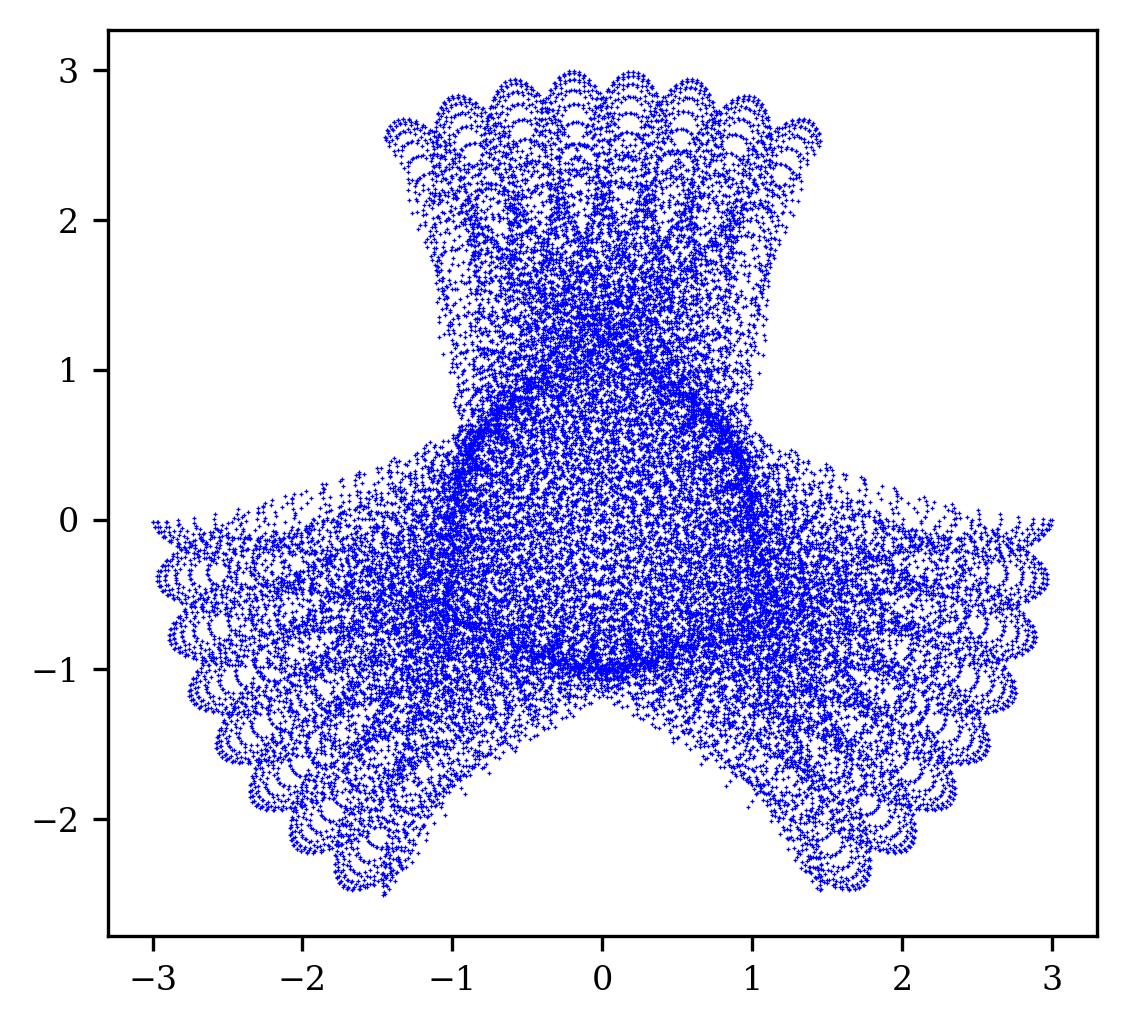}
      \caption{$q= 60383$}
    \end{subfigure}
    \hfill
    \begin{subfigure}[b]{0.32\textwidth}
      \centering
      \includegraphics[width=\textwidth]{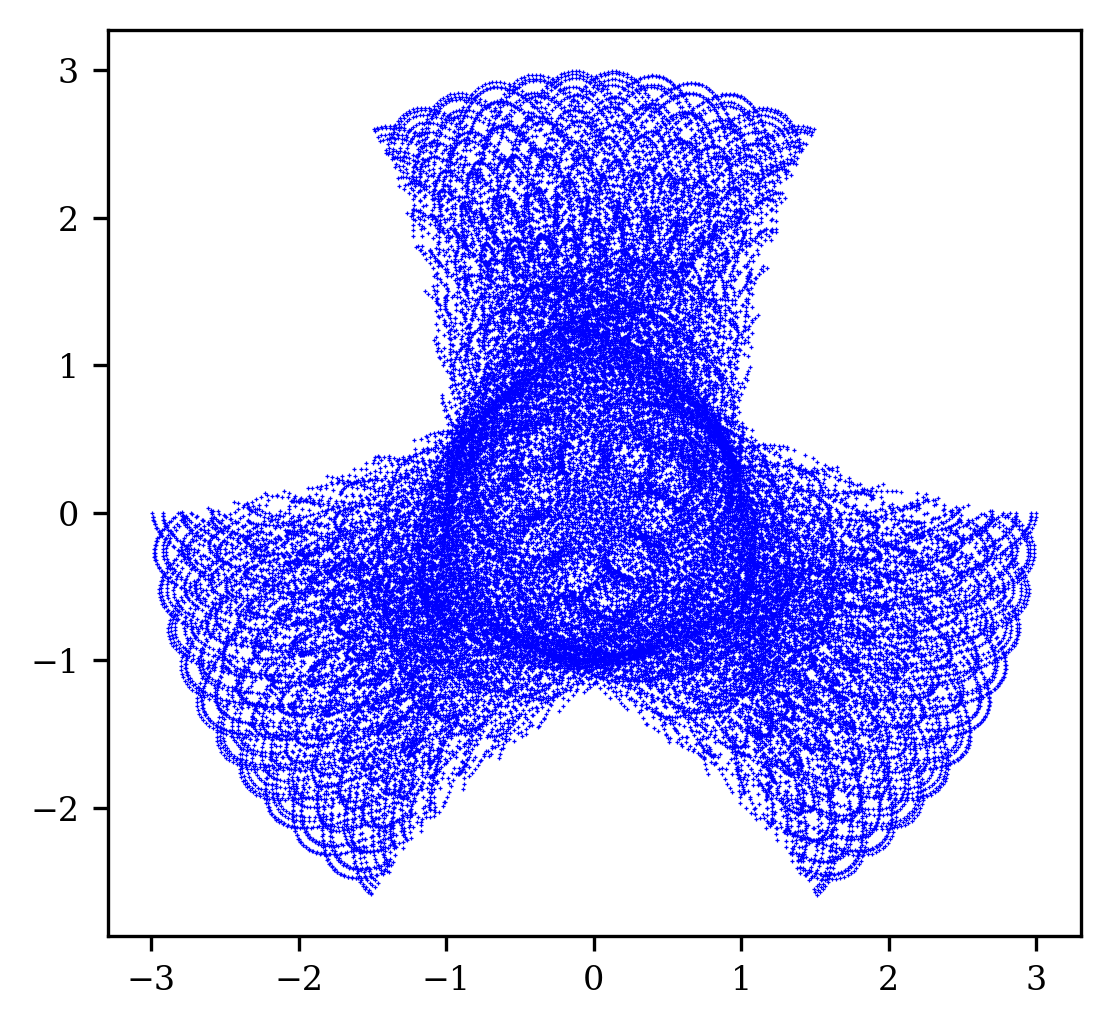}
      \caption{$q=100357$}
    \end{subfigure}
    \caption{The sums \eqref{sum_interval} for $a$ varying in
      $\{ 0, \dots, \frac{q-1}{2}\}$, for three values of $q$}
    \label{different_limit}
  \end{figure}
  \par
  (2) Many examples of uniformly equidistributed sets (modulo primes
  at least) are provided by using the theory of trace functions and
  the Riemann Hypothesis over finite fields. For instance, if
  $f\in \Zz[X]$ is a monic polynomial, then the fractional parts of
  elements of the sets $A_p=f(\Oo_p/p)\subset \Oo_g/p$ are uniformly
  equidistributed. Indeed, one derives, e.g.,
  from~\cite[Prop.\,6.7]{fkm2}, and the Riemann Hypothesis over finite
  fields that $|A_p|\gg |p|$ and that
  $$
  \frac{1}{|A_{p}|} \sum_{\substack{a\in \Oo_g/p\\a=f(b)\text{ for some
        $b$}}} e\Bigl(\frac{ha}{|p|}\Bigr) \ll \frac{1}{|p|^{1/2}}
  $$
  for all~$h\in (\Oo_g/p)^{\times}$, where the implied constant
  depends only on~$\deg(f)$. The simplest example is that of quadratic
  residues.
  \par
  (3) In the last estimate, since the implied constant depends only on
  the degree of the polynomial~$f$, one can take $f$ to depend
  on~$p$. It is natural to ask how large $\deg(f)$ can really be
  taken. The simplest ``test'' case is when $f=X^d$ is a monomial, and
  the question is then whether Proposition~\ref{pr-cond} applies to
  small \emph{multiplicative subgroups}
  $A_{p^n}\subset (\Oo_g/p^n)^{\times}$.

  Using a striking result of Bourgain~\cite{bourgainarbitrary}, and
  adapting an argument of Untrau~\cite[Prop.\,1.14]{untrau} (to show
  that if $\varpi_{p^n}(\gamma(\alpha))\not=0$, then its $|p|$-adic
  valuation is bounded as $p^n$ varies), one can deduce easily that the
  first part of Proposition~\ref{pr-cond} does indeed apply if there
  exists $\delta>0$ such that $A_{p^n}$ is a subgroup of
  $(\Oo_g/p^n)^{\times}$ with $|A_{p^n}|\gg |p|^{n\delta}$.

\end{example}

\section{Additive characters with more general polynomials}
\label{sec-generalizations}

Very simple adaptations of the proof of Proposition~\ref{pr-1} (which
are left to the reader) lead to the following more general statements,
the second of which was also studied by Untrau in the case $g=X^d-1$.

\begin{proposition}[Ultra-short equidistribution, 2]\label{pr-2}
  Let $v\in\Zz[X,X^{-1}]$ be a non-constant Laurent polynomial. Assume
  that $0\notin Z_g$. Define random variables $W_{p^n}$ on $\Oo_g/p^n$
  for $p\in\mathcal{S}_g$ dividing none of the roots of $g$ and
  $n \geqslant 1$, with values in $C(Z_g;\Ss^1)$, by
  $$
  W_{p^n}(a)(x)=e\Bigl(\frac{av(\varpi(x))}{|p|^n}\Bigr).
  $$
  \par
  The random variables $W_{p^n}$ converge in law as $|p|^n\to+\infty$ to the
  random function $W\colon Z_g\to \Ss^1$ such that~$W$ is uniformly
  distributed on the subgroup orthogonal to the abelian group
  $R_{g,v}\subset C(Z_g;\Zz)$ of additive relations between components
  of $(v(x))_{x\in Z_g}$, namely
  $$
  R_{g,v}=\{\alpha\colon Z_g\to\Zz\,\mid\, \sum_{x\in
    Z_g}\alpha(x)v(x)=0\}.
  $$
\end{proposition}

\begin{proposition}[Ultra-short equidistribution, 3]\label{pr-3}
  Let $k\geq 1$ be an integer and fix distinct integers $m_1$, \ldots,
  $m_k$ in~$\Zz$. Assume $0 \notin Z_g$. For $p\in\mathcal{S}_g$ dividing none of the roots
  of $g$ and $n \geqslant 1$, define random variables $Y_{p^n}$ on the
  space $(\Oo_g/p^n)^k$ with uniform probability measure, with values
  in $C(Z_g;\Ss^1)$, by
  $$
  Y_{p^n}(a_1,\ldots, a_k)(x)=e\Bigl(\frac{1}{|p|^n}
  \Bigl(\sum_{i=1}^k a_i\varpi(x)^{m_i}\Bigr)\Bigr)
  $$
  \par
  The random variables $Y_{p^n}$ converge in law as $|p|^n\to+\infty$ to the
  random function $Y\colon Z_g\to \Ss^1$ such that~$Y$ is uniformly
  distributed on the subgroup orthogonal to the abelian group
  $$
  \{\alpha\colon Z_g\to\Zz\,\mid\, \sum_{x\in
    Z_g}\alpha(x)x^{m_i}=0\text{ for } 1\leq i\leq k\}.
  $$
\end{proposition}

As corollaries, we have equidistribution for the sums
$$
\sum_{x\in Z_g(\Ff_q)}e\Bigl(\frac{av(x)}{q}\Bigr)
$$
as $a$ varies in $\Ff_q$ for $q$ totally split in $K_g$ and
$$
\sum_{x\in Z_g(\Ff_q)}e\Bigl(\frac{a_1x^{m_1}+\cdots +a_kx^{m_k}}{q}\Bigr),
$$
as $a_1$, \ldots, $a_k$ vary independently and uniformly in $\Ff_q$ for
$q$ totally split in $K_g$.

\begin{example}
  Consider the case of $g=X^d-1$ and the sums
  \begin{equation}\label{eq-1}
    \sum_{x\in \mmu_d(\Ff_q)}e\Bigl(\frac{a(x+\bar{x})}{q}\Bigr)
  \end{equation}
  and
  \begin{equation}\label{eq-2}
    \sum_{x\in \mmu_d(\Ff_q)}e\Bigl(\frac{ax+b\bar{x}}{q}\Bigr),
  \end{equation}
  as $a$ and $b$ vary in $\Ff_q$ for $q$ totally split in $K_g$.  Both
  satisfy equidistribution,
  but in general have different limiting measures.  For~(\ref{eq-1}),
  we need to determine the functions $\alpha$ satisfying the relation
  $$
  \sum_{x\in\mmu_d}\alpha(x)(x+x^{-1})=0,
  $$
  and for~(\ref{eq-2}), we need to solve
  $$
  \sum_{x\in\mmu_d}\alpha(x)x=\sum_{x\in\mmu_d}\alpha(x)x^{-1}=0.
  $$
  \par
  This last case boils down to the same relations as in
  Section~\ref{sec-examples}, Example~1, since the second sum above is
  the complex-conjugate of the first.

  For~(\ref{eq-1}), on the other hand,
  the relation is equivalent to
  $$
  \sum_{x\in\mmu_d}(\alpha(x)+\alpha(x^{-1}))x=0,
  $$
  which means that $\beta\colon x\mapsto \alpha(x)+\alpha(x^{-1})$
  belongs to the group of additive relations of $X^d -1$.

  We now assume that $d=\ell$ is an odd prime number. Then, by the
  previous examples, the map $\beta$ must be constant. Let then~$\xi$ be
  a non-trivial $\ell$-th root of unity. It is then fairly easy to check
  that the module $R_{X^{\ell}-1,X+X^{-1}}$ is generated by the constant
  function $\alpha_0 = 1$ and the functions $\alpha_j$ for
  $1\leq j\leq (\ell-1)/2$ such that
  $$
  \alpha_j(\xi^k)=\begin{cases}
    0 &\text{ if } k\notin \{j,\ell-j\}\\
    1&\text{ if } k=j\\
    -1&\text{ if } k=\ell-j.
  \end{cases}
  $$

  (It is clear that $\alpha_0$, \ldots, $\alpha_{(\ell-1)/2}$ provide
  relations; conversely, if~$\beta$ is constant then we check that
  $$
  \alpha=\alpha(1)\alpha_0+\sum_{j=1}^{(\ell-1)/2}
  (\alpha(\xi^j)-\alpha(1))\alpha_j,
  $$
  so that these functions generate the group of relations.)
  
  In particular, the module of relations has rank~$(\ell+1)/2$, and the
  limit $W$, in this case, is uniform on the subgroup
  $H_{X^{\ell}-1,X+X^{-1}}$ characterized by
  $f\in H_{X^{\ell}-1,X+X^{-1}}$ if and only if
  $$
  \prod_{j=0}^{\ell-1}f(\xi^j)=1,
  $$
  (corresponding to~$\alpha_0$) and
  $$
  f(\xi^j)=f(\xi^{\ell-j})
  $$
  for $1\leq j\leq (\ell-1)/2$ (corresponding to~$\alpha_j$).

  Consider for instance the case $\ell = 3$. The sums \eqref{eq-2} will
  become equidistributed with respect to the measure on $\Cc$ which is
  the pushforward measure of the uniform measure on $\Ss^1 \times \Ss^1$
  by $ (y_1, y_2) \mapsto y_1 + y_2 + 1/(y_1y_2)$.  This is illustrated
  in Figure \ref{comparaison} (B), since the image of the above map is
  the closed region delimited by a $3$-cusp hypocycloid.
  
  On the other hand, the sums \eqref{eq-1} become equidistributed in
  this case with respect to the image of the Haar measure on~$\Ss^1$ by
  the map $y\mapsto 2y+1/y^2$.  Since the image of this map is precisely
  the $3$-cusp hypocycloid, this explains the picture obtained in Figure
  \ref{comparaison} (A).
  
  \begin{figure}
    \centering
    \begin{subfigure}{5cm}
      \includegraphics[width=\textwidth]{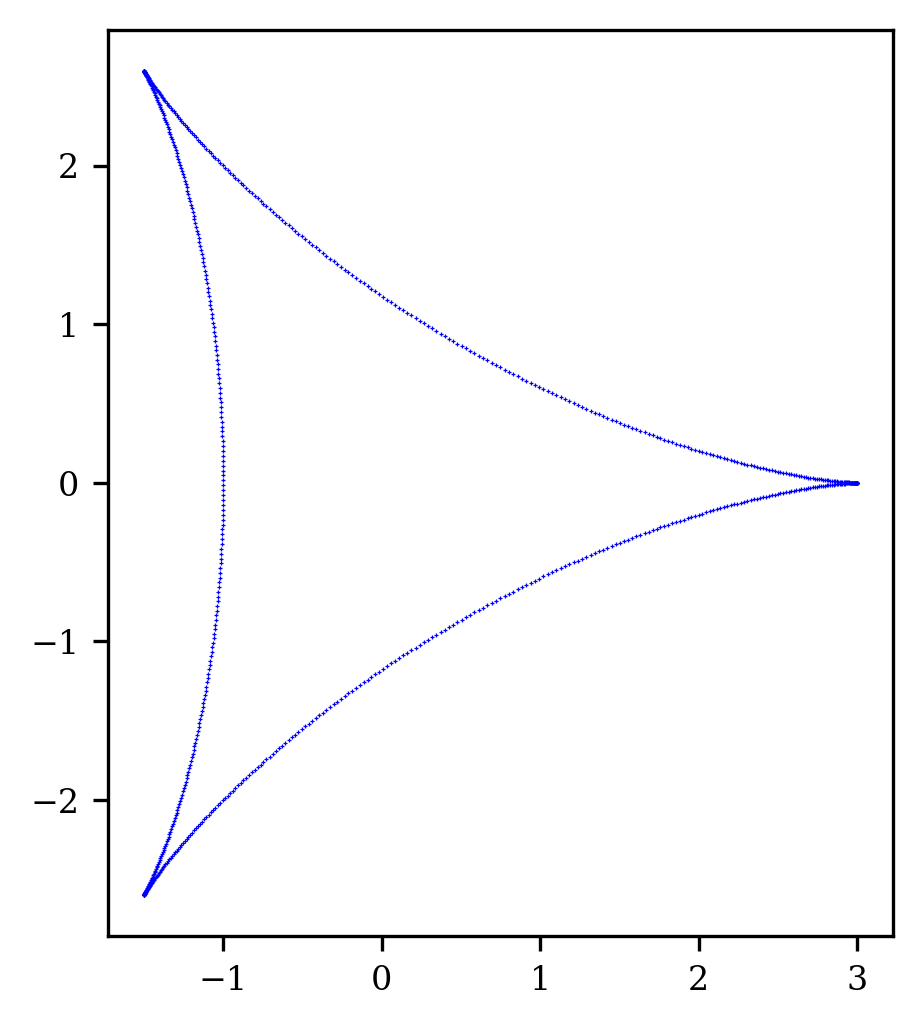}
      \caption{The sums of type \eqref{eq-1} for $d=3$, $q = 811$, and $a$ varying in $\Ff_q$.}
    \end{subfigure}
    \hspace{0.5cm}
    \begin{subfigure}{5cm}
      \includegraphics[width=\textwidth]{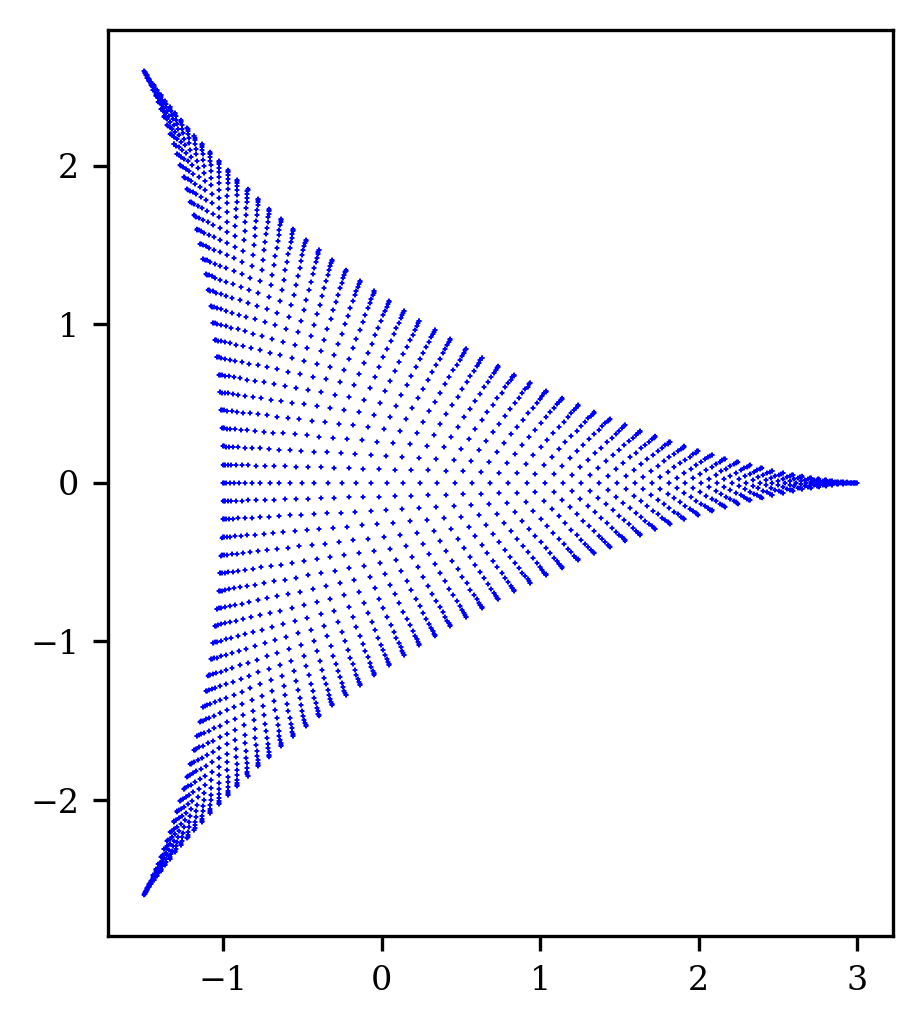}
      \caption{The sums of type \eqref{eq-2} for $d=3$, $q=109$, and $a$ and $b$ varying in $\Ff_q$.}
    \end{subfigure}
    \caption{Comparison between the regions of equidistribution for sums of type \eqref{eq-1} and sums of type \eqref{eq-2}.}
    \label{comparaison}
  \end{figure}
  
  In the case $\ell = 5$, the sums \eqref{eq-2} are equidistributed with
  respect to the measure on $\Cc$ which is the pushforward measure of
  the uniform measure on $(\Ss^1)^4$ by
  $ (y_1, \ldots,y_4) \mapsto y_1 +\cdots + y_4 + 1/(y_1\cdots y_4)$.
  The sums \eqref{eq-1} are equidistributed in this case with respect to
  the image of the Haar measure on~$(\Ss^1)^2$ by the map
  $(y_1,y_2)\mapsto 2y_1+2y_2+1/(y_1y_2)^2$.

  \begin{figure}
    \centering
    \begin{subfigure}{5cm}
      \includegraphics[width=\textwidth]{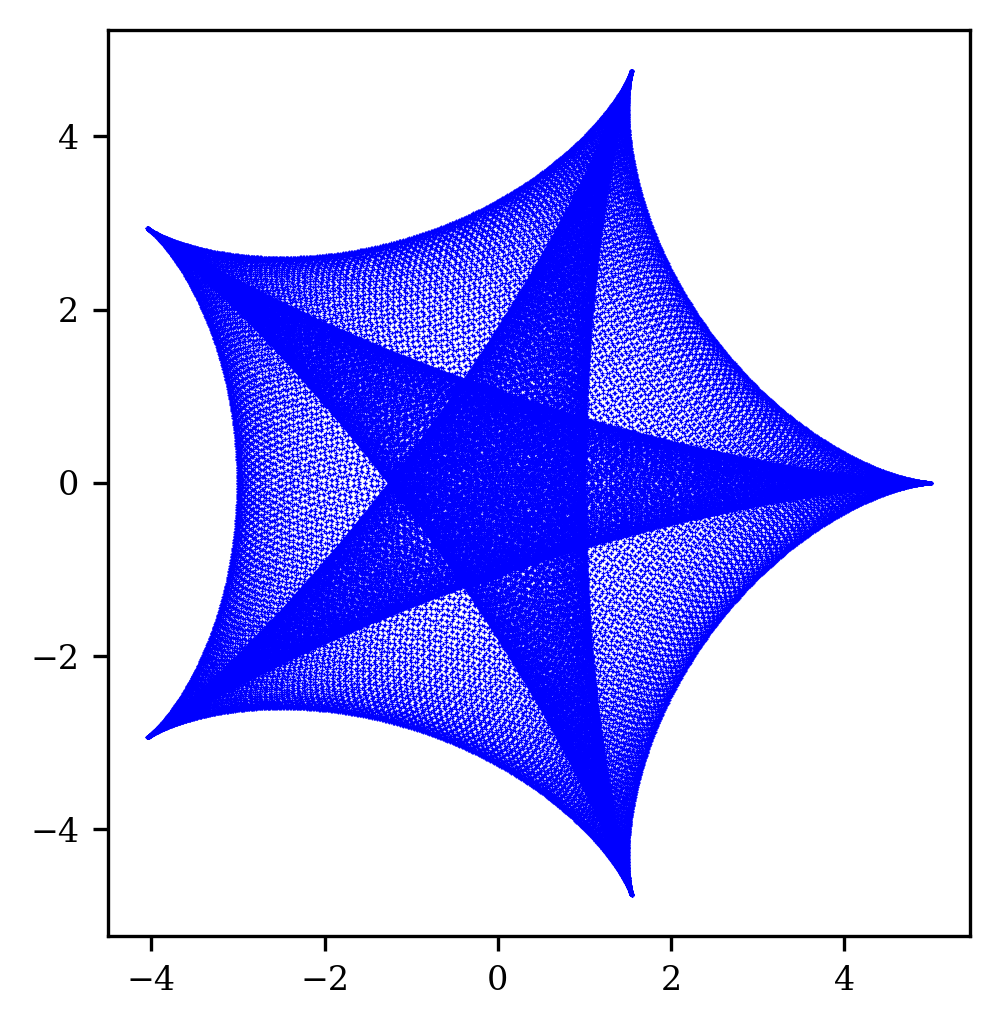}
      \caption{The sums of type \eqref{eq-1} for $d=5$, $q = 96331$, and $a$ varying in $\Ff_q$.}
    \end{subfigure}
    \hspace{0.5cm}
    \begin{subfigure}{5cm}
      \includegraphics[width=\textwidth]{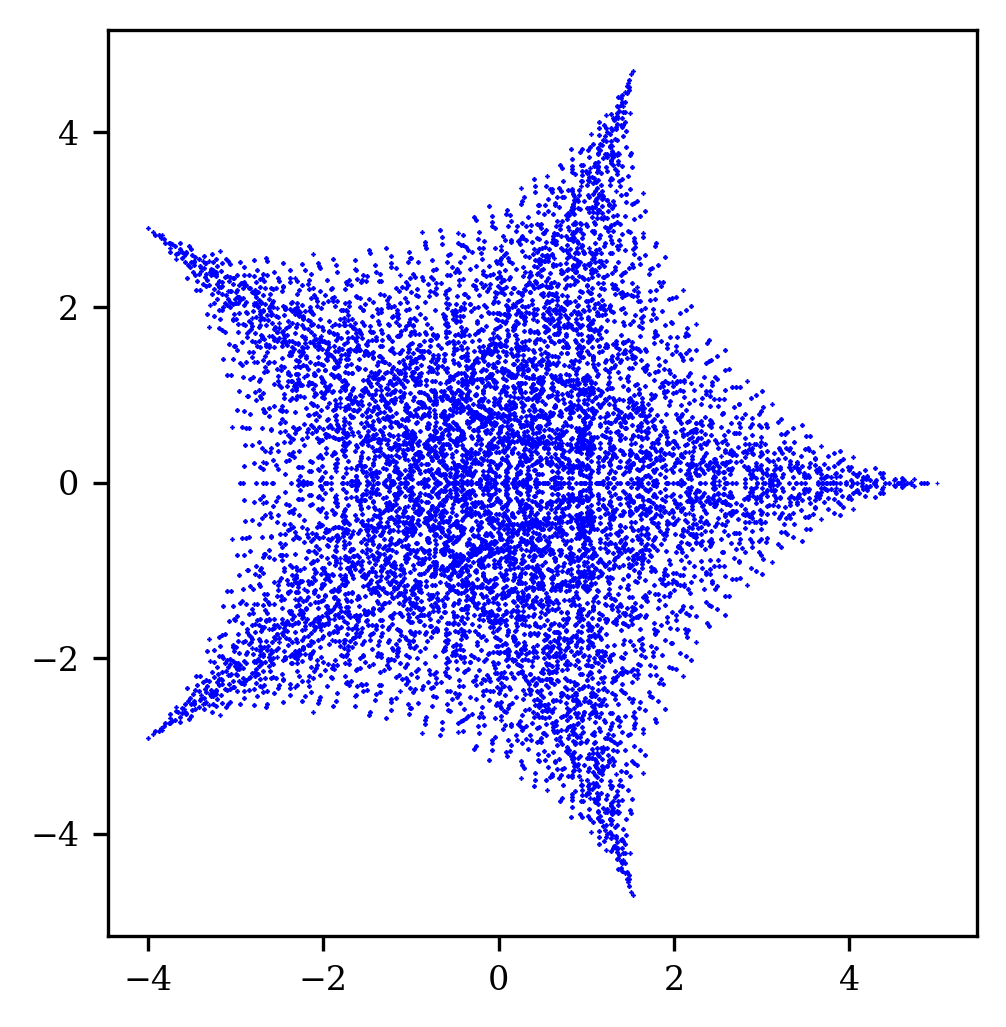}
      \caption{The sums of type \eqref{eq-2} for $d=5$, $q=311$, and $a$ and $b$ varying in $\Ff_q$.}
    \end{subfigure}
    \caption{Comparison between the equidistribution results for sums of type \eqref{eq-1} and sums of type \eqref{eq-2} for $d = 5$.}
    \label{comparaison 2}
  \end{figure}
  
\end{example}

\section{A multiplicative analogue} \label{section multiplicative}

In~\cite{relations}, the group of multiplicative relations between roots
of a polynomial also appears naturally. Is it also relevant for the type
of questions under consideration here?  It turns out that it is, if we
change the probability space, and look at the distribution of sums
$$
\sum_{x\in Z_g}\chi(v(x))
$$
where $\chi$ is a varying multiplicative character of~$\Ff_q$ and $v$ is
a fixed polynomial.

More precisely, we continue with the notation from the previous
section, but assume moreover that $v(x)\not=0$ for $x\in Z_g$ (for
instance, $0\notin Z_g$ if $v=X$). For $p\in\mathcal{S}_g$, we now
consider the probability space $X_p$ of multiplicative characters
$\chi\colon (\Oo_g/p)^{\times}\to \Ss^1$, with the uniform probability
measure (we consider only primes instead of prime powers for
simplicity here).  The random variables are now $\widetilde{U}_p$,
taking values in the group $C(Z_g;\Ss^1)$, and defined by
$$
\widetilde{U}_p(\chi)(x)=\chi(v(\varpi(x))).
$$

\begin{proposition}
  The random variables $\widetilde{U}_p$ converge in law as
  $|p|\to+\infty$ to the random function
  $\widetilde{U}\colon Z_g\to \Ss^1$ such that~$\widetilde{U}$ is
  uniformly distributed on the subgroup
  $\widetilde{H}_g\subset C(Z_g;\Ss^1)$ which is orthogonal to the
  abelian group $\widetilde{R}_g\subset C(Z_g;\Zz)$ of multiplicative
  relations between values of $v$ on~$Z_g$, namely we have
  $$
  \widetilde{R}_{g,v}=\{\alpha\colon Z_g\to\Zz\,\mid\, \prod_{x\in
    Z_g}v(x)^{\alpha(x)}=1\},
  $$
  and
  $$
  \widetilde{H}_{g,v}=\{f\in C(Z_g;\Ss^1)\,\mid\, \text{ for all
    $\alpha\in \widetilde{R}_d$, we have } \prod_{x\in
    Z_g}f(x)^{\alpha(x)}=1\}.
  $$
  \par
  In particular, as $q\to+\infty$ among primes totally split in $K_g$,
  the sums
  $$
  \sum_{x\in Z_g}\chi(v(x))
  $$
  converge in law to the image by the linear form $\sigma$ of the Haar
  probability measure on~$\widetilde{H}_{g,v}$.
\end{proposition}

\begin{proof}
  This is the same as Proposition~\ref{pr-1}, mutatis mutandis, with now
  $$
  \expect(\eta(\widetilde{U}_p))= \frac{1}{|p|-1} \sum_{\chi\in X_p}
  \prod_{x\in Z_g}\chi(v(\varpi(x)))^{\alpha(x)}
  $$
  for a character $\eta$ of $C(Z_g;\Ss^1)$ determined by the function
  $\alpha$. This is
  $$
  \expect(\eta(\widetilde{U}_p))= \frac{1}{|p|-1} \sum_{\chi\in X_p}
  \chi\Bigl(\varpi\Bigl(\prod_{x\in Z_g}v(x)^{\alpha(x)}\Bigr)\Bigr)
  $$
  and for the same reasons as before, converges to~$1$ or~$0$, depending
  on whether
  $$
  \prod_{x\in Z_g}v(x)^{\alpha(x)}
  $$
  is equal to~$1$ or not. 
\end{proof}

\begin{example}
  (1) Here also there are some interesting examples in~\cite{relations}
  and~\cite{repr} if we take $v=X$ (so that $R_{g,v}$ corresponds to
  multiplicative relations between roots of~$g$). In particular, we
  could take a polynomial $g$ with Galois group the Weyl group of
  $\mathbf{E}_8$, which is of degree~$248$ but has all roots obtained
  multiplicatively from $8$ of them (see~\cite{e8} for examples).
  \par
  (2) For $v=X$ again, the case of $g=X^d-1$ is quite
  degenerate. Indeed, for $q\equiv 1\mods{d}$ and a multiplicative
  character $\chi$ of~$\Ff_q$, the sum
  $$
  \sum_{x\in\mmu_d(\Ff_q)}\chi(x)
  $$
  is either~$d$ or~$0$, depending on whether the character~$\chi$ is
  trivial on the $d$-th roots of unity or not. The former means that
  $\chi^{(|p|-1)/d}=1$, and there are therefore $(|p|-1)/d$ such
  characters. Hence the sum is equal to $d$ with probability $1/d$, and
  to~$0$ with probability $1-1/d$.
  \par
  (3) If we consider the class polynomial for CM curves (as in
  Section~\ref{sec-examples}, Example~4), we are led to consider
  potential multiplicative relations between $j$-invariants. This is
  apparently more challenging than the additive case, and we do not have
  a precise answer at the moment (see, e.g., the papers of Bilu, Luca
  and Pizarro-Madariaga~\cite{b-l-p} and Fowler~\cite{fowler} for
  partial results).
\end{example}

\section{Higher rank trace functions}\label{sec-proofs2}

We now elaborate on the setting of Section~\ref{sec-proofs1} to involve
more general trace functions. Thus the goal is to study the distribution
of
$$
\sum_{x\in Z_g(\Oo_g/p)}t_p(ax),\quad\quad\text{ or }
\quad\quad
\sum_{x\in Z_g(\Oo_g/p)}t_p(a+x),
$$
(or other similar expressions) when $t_p$ is, for
each~$p\in\mathcal{S}_g$, a trace function over the finite field
$\Oo_g/p$. The cases of Section~\ref{sec-proofs1} correspond to
$t_p(x)=e(x/|p|)$ or $t_p(x)=e(v(x)/|p|)$, i.e., to the trace functions of
Artin--Schreier sheaves.

We thus assume that for each $p\in\mathcal{S}_g$, we are given a
middle-extension sheaf $\mcF_p$ on the affine line over $\Oo_g/p$. We
assume that these sheaves are pure of weight~$0$, and have the same
rank~$r$, and moreover have bounded conductor in the sense of Fouvry,
Kowalski and Michel~\cite{fkm1, sop}.

We denote by~$\Un_r(\Cc)^{\sharp}$ the space of conjugacy classes in the
unitary group~$\Un_r(\Cc)$.  For any~$x\in\Oo_g/p$ such that~$\mcF_p$ is
lisse at~$x$, the action of the geometric Frobenius automorphism at~$x$
on the stalk of~$\mcF_p$ at~$x$ gives a unique conjugacy class
$\Theta_p(x)\in\Un_r(\Cc)^{\sharp}$. We denote
\begin{align*}
  A_p&=\{a\in (\Oo_g/p)^{\times} \,\mid\, \text{ for all } x\in
  Z_g(\Oo_g/p),\text{ $\mcF_p$ is lisse at $ax$}\},
  \\
  B_p&=\{a\in\Oo_g/p\,\mid\, \text{ for all } x\in Z_g(\Oo_g/p),\text{
    $\mcF_p$ is lisse at $a+x$}\}.
\end{align*}

Note that $|A_p|\gg |p|$ if $0\notin Z_g$ and $|B_p|\gg |p|$ in all
cases.

We can define random functions $U_p$ and $V_p$ on $A_p$ and $B_p$,
respectively (with the uniform probability measure), with values in the
space $C(Z_g;\Un_r(\Cc)^{\sharp})$ by
$$
U_p(a)(x)=\Theta_p(ax),\quad\quad
V_p(a)(x)=\Theta_p(a+x).
$$
\par
Since the trace function $t_p$ of $\mcF_p$ satisfies
$$
t_p(x)=\Tr(\Theta_p(x))
$$
when $\mcF_p$ is lisse at~$x$, we see that if one can prove that $(U_p)$
or $(V_p)$ have a limit, then the corresponding sums
\begin{equation}\label{eq-sums-higher-rank}
  \sum_{x\in Z_g(\Oo_g/p)}t_p(ax),\quad\quad\text{ and/or }
  \quad\quad\sum_{x\in Z_g(\Oo_g/p)}t_p(a+x),
\end{equation}
for $a\in A_p$ (resp. $B_p$) will become equidistributed according to
the image of this limit distribution by the map
$$
f\mapsto \sum_{x\in Z_g}\Tr(f(x))
$$
for $f\colon Z_g\to \Un_r(\Cc)^{\sharp}$.

\begin{remark}
  It happens frequently that $t_p(y)\ll p^{-1/2}$ if $\mcF_p$ is not
  lisse at~$y$, where the implied constant depends only on the conductor
  of~$\mcF_p$. In such a csae, the equidistribution for the
  sums~(\ref{eq-sums-higher-rank}) holds when~$a$ is taken in all of
  $(\Oo_g/p)^{\times}$ or $\Oo_g/p$, since the remaining value of~$a$
  have negligible contributions.
\end{remark}

We obtain a large supply of examples from known results on estimates of
``sums of products'' of trace functions (see~\cite{sop}). Although the
terminology might not be familiar to all readers, examples after the
proof will provide concrete illustrations.

\begin{proposition}
  Assume that $\mcF_p$ is bountiful in the sense of~\cite{sop} for all
  $p$ in $\mathcal{S}_g$. 
  \par
  \emph{(1)} If $\mcF_p$ is of $\Sp_{r}$-type for all $p$, then $(U_p)$
  and $(V_p)$ converge in law as $|p|\to+\infty$, with limit uniform on
  $C(Z_g;\USp_r(\Cc)^{\sharp})$.
  \par
  \emph{(2)} If $\mcF_p$ is of $\SL_{r}$-type for all $p$, and the
  special involution, if it exists, is not $y\mapsto -y$, then $(U_p)$
  and $(V_p)$ converge in law as $|p|\to+\infty$, with limit uniform on
  $C(Z_g;\SU_r(\Cc)^{\sharp})$.
  \par
  \emph{(3)} If $\mcF_p$ is of $\SL_{r}$-type for all $p$ with special
  involution $y\mapsto -y$, then $(V_p)$ converge in law as
  $|p|\to+\infty$ with limit uniform on $C(Z_g;\SU_r(\Cc)^{\sharp})$, and
  $(U_p)$ converges in law with limit uniform on
  $$
  \{f\colon Z_g \to \SU_r(\Cc)\,\mid\, f(x)=\overline{f(y)}\text{ if }
  x=-y \}.
  $$
  \par
  In all three cases, we assume that $0\notin Z_g$ in the case
  of~$(U_p)$.
\end{proposition}

\begin{proof}
  We argue with $U_p$, as the case of $V_p$ is very similar. By
  definition, the random variables $U_p$ take values in
  $C(Z_g;\USp_r(\Cc)^{\sharp})$. Applying the Weyl Criterion, it
  suffices to show that if $(\pi_x)_{x\in Z_g}$ is a family of
  irreducible representations of $\USp_r(\Cc)$, not all trivial, with
  characters $\chi_x=\Tr(\pi_x)$, we have
  $$
  \lim_{|p|\to +\infty} \frac{1}{|p|}\sum_{a\in \Oo_g/p}\  \prod_{x\in
    Z_g(\Oo_g/p)}\chi_x(\Theta_p(ax))=0.
  $$
  \par
  The sum is, up to negligible amount coming from points
  where~$\mcF_p$ is not lisse, the sum of the traces of Frobenius on
  the sheaf
  $$
  \mcG=\bigotimes_{x\in Z_g(\Oo_g/p)} \pi_x([a\mapsto ax]^*\mcF_p),
  $$
  and by Riemann Hypothesis over finite fields of Deligne, we obtain
  $$
  \frac{1}{|p|}\sum_{a\in \Oo_g/p}\  \prod_{x\in
    Z_g(\Oo_g/p)}\chi_x(\Theta_p(ax))\ll |p|^{-1/2}
  $$
  as soon as the geometric monodromy group of this sheaf has no trivial
  subrepresentation in its standard representation. This is true because
  the bountiful property of $\mcF_p$ ensures that the geometric
  monodromy group of $\mcG$ is the product group $\prod_x \Sp_r$.
  \par
  The argument is similar for (2); for (3), we have to take into account
  the fact that the assumption implies that $[a\mapsto -a]^*\mcF_p$ is
  isomorphic to the dual of~$\mcF_p$, so that
  $\Tr(\Theta_p(-ax))=\overline{\Tr(\Theta_p(ax))}$ for all $x\in Z_g$.
\end{proof}

\begin{example}
  We illustrate here all three cases with examples.
  \par
  (1) The classical Kloosterman sums $\Kl_2$ (as in part (2) of
  Theorem~\ref{th-1}) are trace functions of a bountiful sheaf of
  rank~$r=2$ of $\Sp_2$-type which is lisse except at~$0$
  and~$\infty$. Thus the first case of the proposition applies, and in
  particular this establishes the second part of Theorem~\ref{th-1}, in
  view of the fact that the trace of a uniform random matrix
  in~$\SU_2(\Cc)$ is Sato--Tate distributed.
  \par
  Similarly, for even-rank hyper-Kloosterman sums (for which $\mcF_p$ is
  also lisse except at~$0$ and~$\infty$), we obtain the $\USp_r$ case
  (see~\cite[\S\,3.2]{sop}).
  \par
  (2) If $r$ is odd, then the hyper-Kloosterman sum $\Kl_r(a;p)$ arise
  as trace functions of a bountiful sheaf of $\SL_r$-type with special
  involution $y\mapsto -y$ (see~\cite[\S\,3.3]{sop}), which is lisse
  except at~$0$ and~$\infty$. So the third case of the proposition
  applies here. In particular, if the polynomial $g$ is even or odd (so
  that $Z_g=-Z_g$), the support of the limit of $U_p$ is only
  ``half-dimensional''.
  \par
  (3) Examples of trace functions coming from bountiful sheaves of
  $\SL_r$-type without special involution are given for instance by
  $$
  t_p(x)=\frac{1}{\sqrt{|p|}}
  \sum_{y\in \Oo_g/p}\chi(h(y))e\Bigl(\frac{xy}{|p|}\Bigr)
  $$
  where $h\in\Zz[X]$ is a ``generic'' squarefree polynomial of degree
  $\geq 2$. This follows from~\cite[Prop.\,3.7]{sop}, where the meaning
  of ``generic'' is also explained; here also, the sheaf $\mcF_p$ is
  lisse except at~$0$ and~$\infty$.
\end{example}

\end{document}